\documentclass[12pt]{amsart}
\DeclareRobustCommand{\gobblefive}[5]{}
\newcommand*{\SkipTocEntry}{\addtocontents{toc}{\gobblefive}}
\usepackage[margin=3.3cm]{geometry}
\usepackage[dvipdfmx]{hyperref,graphicx}
\usepackage{tikz}
\usepackage{amscd,amsmath,amsfonts,amssymb,latexsym}
\usepackage{hyperref}
\usepackage{xcolor}
\hypersetup{
    colorlinks=true,
    citecolor=blue,
    linkcolor=blue,
    urlcolor=blue,
}
\usepackage[all]{xy}
\usepackage{slashed}
\usepackage{soul}
\usepackage{comment}
\usepackage{extarrows}
\usepackage{mathtools}
\usepackage{url}
\usepackage{lipsum}

\numberwithin{equation}{section}
\theoremstyle{plain}
\newtheorem{lemma}{Lemma}[section]
\newtheorem{proposition}[lemma]{Proposition}

\newtheorem{theorem}[lemma]{Theorem}
\newtheorem{corollary}[lemma]{Corollary}

\theoremstyle{definition}
\newtheorem{definition}[lemma]{Definition}
\newtheorem{remark}[lemma]{Remark}

\let\C\relax
\newcommand{\C}{{\mathbb C}}
\newcommand{\R}{{\mathbb R}}
 
\newcommand{\N}{{\mathbb N}}
\newcommand{\Hh}{{\mathcal H}}
\newcommand{\Ll}{{\mathcal L}}
\newcommand{\Mm}{{\mathcal M}}

\newcommand{\id}{{\rm id}}
\newcommand{\Om}{{\Omega}}
\newcommand{\om}{{\omega}}

\newcommand{\la}{\langle}
\newcommand{\ra}{\rangle}
\newcommand{\p}{{\partial}}
\newcommand{\bp}{{\bar{\partial}}}
\newcommand{\vol}{\mbox{\rm vol}}

\newcommand{\Sp}{{\text{\rm Spin}(7)}}

\renewcommand{\Im}{{ \rm Im \,}}

\newcommand{\Aa}{{\mathcal A}}
\newcommand{\Bb}{{\mathcal B}}
\newcommand{\Ff}{{\mathcal F}}
\newcommand{\wFf}{{\widehat{\mathcal F}_{G_2}}}
\newcommand{\Gg}{{\mathcal G}}
\newcommand{\Kk}{{\mathcal K}}
\newcommand{\Pp}{{\mathcal P}}
\newcommand{\Ss}{{\mathcal S}}
\newcommand{\Tt}{{\mathcal T}}

\newcommand{\g}{{\mathfrak g}}
\renewcommand{\i}{{\sqrt{-1}}}
\renewcommand{\l}{{\ell}}

\newcommand{\n}{{\nabla}}

\begin{document}
\title[Deformation of dHYM and dDT]
{Deformation theory of deformed Hermitian Yang--Mills connections and deformed Donaldson--Thomas connections}
\author{Kotaro Kawai}
\address{Department of Mathematics, Faculty of Science, Gakushuin University, 1-5-1 Mejiro, Toshima-ku, Tokyo 171-8588, Japan}
\email{kkawai@math.gakushuin.ac.jp}
\author{Hikaru Yamamoto}
\address{Department of Mathematics, Faculty of Pure and Applied Science, University of Tsukuba, 1-1-1 Tennodai, Tsukuba, Ibaraki 305-8577, Japan}
\email{hyamamoto@math.tsukuba.ac.jp}
\thanks{The first named author is supported by 
JSPS KAKENHI Grant Number JP17K14181 and Research Grants of Yoshishige Abe Memorial Fund, and the second named author is supported by JSPS KAKENHI Grant Number JP18K13415 and Osaka City University Advanced Mathematical Institute (MEXT Joint Usage/Research Center on Mathematics and Theoretical Physics)}
\begin{abstract}
A deformed Donaldson--Thomas (dDT) connection is a Hermitian connection of a Hermitian line bundle over a $G_2$-manifold $X$ satisfying 
a certain nonlinear PDE. This is considered to be the mirror of a (co)associative cycle in the context of mirror symmetry. 
The dDT connection is an analogue of a deformed Hermitian Yang--Mills (dHYM) connection which is extensively studied recently. 

In this paper, we study the moduli spaces of dDT and dHYM connections. 
In the former half, we prove that 
the deformation of dDT connections is controlled by a subcomplex of the canonical complex, 
an elliptic complex defined by Reyes Carri\'on, by introducing a new coclosed $G_2$-structure. 
If the deformation is unobstructed, we also show that the connected component of the moduli space is a $b^{1}$-dimensional torus,  
where $b^{1}$ is the first Betti number of $X$. 
A canonical orientation on the moduli space is also given. 
We also prove that the obstruction of the deformation vanishes if we perturb the $G_2$-structure 
generically under some assumptions. 

In the latter half, we prove that 
the moduli space of dHYM connections, if it is nonempty, is a $b^{1}$-dimensional torus, especially, it is connected and orientable. 
We also prove the existence of a family of moduli spaces along a deformation of underlying structures 
if two cohomology classes vanish. 
\end{abstract}
\maketitle
\tableofcontents
\section{Introduction}\label{sec1}
A \emph{deformed Donaldson--Thomas (dDT) connection} is a Hermitian connection $\nabla$ of 
a smooth complex line bundle $L$ with a Hermitian metric $h$ over 
a manifold $X^7$ with a $G_2$-structure $\varphi \in \Om^3$
satisfying 
\[
\frac{1}{6} F_\nabla^3 + F_\nabla \wedge \ast\varphi=0,  
\]
where $F_\n$ is the curvature 2-form of $\n$. We consider it as a $\i \R$-valued closed 2-form on $X$. 
This is an analogue in $G_{2}$-geometry of dHYM connections defined as follows. 
A \emph{deformed Hermitian Yang--Mills (dHYM) connection} is a Hermitian connection $\nabla$ of 
a smooth complex line bundle $L$ with a Hermitian metric $h$ over a K\"ahler manifold 
$(X,\omega)$ of $\dim_\C X=n$ satisfying 
\[
F^{0,2}_{\nabla}=0
\quad\mbox{and}\quad
\mathop{\mathrm{Im}}\left(e^{-\i \theta} (\omega + F_\nabla)^n\right)=0
\]
for a constant $\theta\in\mathbb{R}$, where $F_{\nabla}^{0,2}$ is the $(0,2)$-part 
of the curvature 2-form $F_{\nabla}$ of $\nabla$. 
The former condition is usually interpreted as the integrability condition 
and the complex number $e^{\sqrt{-1}\theta}$ is called the phase. 

\SkipTocEntry \subsection*{Background}
Before introducing main theorems, we provide the background of these subjects. 
In theoretical physics, dHYM connections were derived as minimizers, called BPS states, of a functional, 
called the Dirac-Born-Infeld (DBI) action, by Mari\~no, Minasian, Moore and Strominger \cite{MMMS}. 
See also an explanation by Collins, Xie and Yau \cite{CXY}. 
At the same time, in mathematics, 
the dHYM connection was found by Leung, Yau and Zaslow \cite{LYZ}, 
slightly after the discovery of a similar notion by Leung \cite{L}. 
It was obtained 
as the real Fourier--Mukai transform of a graphical special Lagrangian submanifold 
with a flat connection over it (a special Lagrangian cycle) in a Calabi--Yau manifold 
in the context of mirror symmetry. 
Similarly, the dDT connection was introduced by Lee and Leung \cite{LL} 
as the real Fourier--Mukai transform of 
an associative cycle, an associative submanifold with a flat connection over it, 
or a coassociative cycle, 
a coassociative submanifold with an ASD connection over it. 
Note that each of associative and coassociative cycles 
corresponds to the dDT connection. 
\SkipTocEntry \subsection*{Motivation}
As mentioned above,  dHYM and dDT connections 
are the mirror objects of special Lagrangian and (co)associative cycles, respectively. 
As the names indicate, dHYM and dDT connections can also be considered as 
analogies of Hermitian Yang--Mills (HYM) connections 
and Donaldson--Thomas (DT) connections, respectively. 
HYM and DT connections are also called Hermite--Einstein connections and $G_2$-instantons, respectively. 
Thus, it is natural to expect that dHYM and dDT connections would have similar properties to these objects. 
In this paper, we especially study whether the moduli spaces has similar properties. 
\SkipTocEntry \subsection*{Main Results and implication}
We now state our main results of this paper. 
\begin{theorem}[Theorems \ref{thm:moduli MG2}, \ref{newtheorem}, \ref{thm:moduli MG2 generic}
and Corollary \ref{cor:oriG2}] \label{Main2}
Let 
$X^7$ be a compact 7-manifold with a coclosed $G_2$-structure $\varphi \in \Om^3$ 
and $L$ be a smooth complex line bundle with a Hermitian metric $h$ over $X$. 
Let $\Mm_{G_2}$ be the set of all dDT connections of $L$  divided by the $U(1)$-gauge action, 
and call it the moduli space. Suppose that $\Mm_{G_2} \neq \emptyset$. 
\begin{enumerate}
\item
If $H^2(\#_\nabla) = \{\, 0 \,\}$ for $\nabla \in \Mm_{G_2}$, 
where $H^2(\#_\nabla)$ is the second cohomology of the complex $(\#_\nabla)$ 
for $\nabla$ defined in Subsection \ref{sec:infi deform G2}, 
then the connected component of $[\nabla]$ in the moduli space $\Mm_{G_2}$ is a smooth manifold 
which is homeomorphic to a $b^{1}$-dimensional torus, 
where $b^1$ is the first Betti number of $X$.
\item
Recall that we can call a 4-form 
that is pointwisely identified with \eqref{varphi*} a $G_2$-structure 
and dDT connections are defined by a $G_2$-structure considered as a 4-form. 
Suppose that 
$\n$ is a dDT connection (with respect to $\varphi$ (or $* \varphi$)) 
and one of the following conditions hold.
\begin{enumerate}
\item
The $G_2$-structure $\varphi \in \Om^3$ is torsion-free or nearly parallel. 
\item
The connection $\n$ satisfies $F_\n^3 \neq 0$ on a dense set of $X$. 
\end{enumerate}
Then, for every generic $G_2$-structure 
$\psi' \in \Om^4$ which satisfies $d \psi'=0$, $[\psi'] = [* \varphi] \in H^4_{dR}$ and is close to $* \varphi$, 
and also for every $[\nabla']\in \mathcal{M}_{G_2, \psi'}$ sufficiently close to $[\nabla]$, 
where $\mathcal{M}_{G_2, \psi'}$ is the moduli space of dDT connections for $\psi'$, 
the connected component of $[\nabla']$ in $\mathcal{M}_{G_2, \psi'}$ is a smooth manifold which is homeomorphic to a $b^{1}$-dimensional torus. 
\item
Suppose that $H^2(\#_\nabla) = \{\, 0 \,\}$ for any $[\n] \in \Mm_{G_2}$. 
Then, $\Mm_{G_2}$ is a $b^{1}$-dimensional manifold which is homeomorphic to the disjoint union of tori 
and it admits a canonical orientation determined by an orientation of $\det D$, 
where $D$ is a Fredholm operator defined by \eqref{cancpxfordDT G23}. 
\end{enumerate}
\end{theorem}

Note that the complex $(\#_\nabla)$ 
can be regarded as a subcomplex of the canonical complex 
introduced by Reyes Carri\'on \cite{Reyes}. 
In Theorem \ref{Main2}, 
the $G_2$-structure $\varphi$ need not to be torsion-free. 
It only has to be coclosed,  which implies that 
there are many explicit examples 
such as the standard 7-sphere $S^7$ for which we can apply this theorem. 
As in the following Theorem \ref{Main1} for dHYM connections, 
one might expect that $\Mm_{G_2}$ itself is a torus, in other words, 
every deformation is unobstructed and $\Mm_{G_2}$ is connected. 
However, this does not hold in general because 
there are examples of obstructed dDT connections and 
moduli spaces which contain at least two connected components of different dimensions given in \cite[Theorem 1.10]{LO}. 

One might think that (3) is immediate from (1) 
because (1) implies that each connected component of $\Mm_{G_{2}}$ is a torus, and hence, 
 $\Mm_{G_{2}}$ is orientable. 
The point in (3) is that we can choose an orientation of $\Mm_{G_{2}}$ canonically. 
As pointed out in \cite[Section 5.4]{DK}, 
if, for example, there are $N$ connected components in $\Mm_{G_{2}}$, there are $2^{N}$ orientations on $\Mm_{G_{2}}$ a priori. 
However, (3) states that we can pick the canonical one from that $2^{N}$ orientations. 
A canonical choice of orientations will be useful to define enumerative invariants 
such as Casson invariants of 3-manifolds or Donaldson invariants of 4-manifolds. 
It is because we must count ``with signs'' to define enumerative invariants 
and we need a canonical orientation of the moduli space to determine the signs. 

By \cite{McLean}, deformations of coassociative submanifolds are unobstructed 
while those of associative submanifolds are obstructed in general. 
The moduli space of associative submanifolds 
is a 0-dimensional manifold for generic $G_2$-structures 
by \cite{Gayet} and admits a canonical orientation by \cite{Joyce}. 
Similarly, 
the deformations of $G_2$-instantons are obstructed in general and  
the moduli space is a 0-dimensional manifold for generic $G_2$-structures.  
It admits a canonical orientation by \cite{JTU, JU, Up, Wa}. 
Hence, Theorem \ref{Main2} indicates that 
the moduli space $\Mm_{G_2}$ of dDT connections 
is similar to each of that of associative submanifolds and that of $G_2$-instantons. 
Moreover, if $b^1=0$, which includes the case that 
a $G_2$-manifold $(X^7, \varphi)$ has full holonomy $G_2$, 
the expected dimension of $\Mm_{G_2}$ is 0, 
which also agrees with these two cases. 

A challenging problem in $G_2$-geometry is 
to define an enumerative invariant of $G_2$-manifolds 
by counting associative submanifolds or $G_2$-instantons. 
By the similarities of moduli spaces, 
we might hope to define an enumerative invariant 
by counting dDT connections. \\

Although the main topic of this paper is the moduli space of dDT connections on a $G_{2}$-manifold, 
we give a few results on the moduli space of dHYM connection on a K\"ahler manifold which are discovered in the process of this study. 

\begin{theorem}[Theorems \ref{thm:smooth M} and \ref{deformdHYM}]\label{Main1}
Let $(X,\omega)$ be a compact K\"ahler manifold and 
$L$ be a smooth complex line bundle with a Hermitian metric $h$ over $X$. 
Let $\mathcal{M}$ be the set of all dHYM connections of $L$ with phase $e^{\sqrt{-1}\theta}$ divided by the $U(1)$-gauge action, and call it the moduli space. 

\begin{enumerate}
\item
If $\mathcal{M}\neq\emptyset$, then the moduli space $\mathcal{M}$ is homeomorphic to 
a $b^{1}$-dimensional torus, where $b^{1}$ is the first Betti number of $X$. 
Especially, it is orientable. 
\item
Let $B\subset \mathbb{R}^{m}$ be an open ball with $0\in B$ 
and assume that a smooth family of K\"ahler structures $\{\,(\omega_{t},g_{t},J_{t})\mid t\in B\,\}$ on $X$, 
Hermitian metrics $\{\,h_{t}\mid t\in B\,\}$ of $L$ and constants $\{\,\theta_{t}\in\mathbb{R}\mid t\in B\,\}$ are given. 
Suppose that there exists a dHYM connection $\nabla_0$ 
of $(L,h_{0})$ on $(X,\omega_{0},g_{0},J_{0})$ with phase $e^{\sqrt{-1}\theta_{0}}$ satisfying 
\begin{equation*}
\left\{
\begin{aligned}
&[(F_{\nabla_0})^{(0,2)_{t}}]=0\quad\mbox{in }H^{0,2}_{\bar{\partial}_{t}},\\
&\left[\mathop{\mathrm{Im}}\left(e^{-\sqrt{-1}\theta_{t}}(\omega_{t}+F_{\nabla_0})^{n}\right)\right]=0\quad\mbox{in }H^{2n}_{dR}, 
\end{aligned}
\right.
\end{equation*}
where $(0,2)_{t}$ is the $(0,2)$-part with respect to $J_{t}$ and 
$H^{0,2}_{\bar{\partial}_{t}}:=H^{0,2}_{\bar{\partial}_{t}}(X,J_{t})$ is 
the Dolbeault cohomology defined by the complex structure $J_{t}$. 

Then, there exists an open set $B'\subset B$ containing $0$ such that $\mathcal{M}_{B'}$ is a $T^{b^{1}}$-bundle over $B'$, 
where $\Mm_{B'} = \cup_{t\in B'}\Mm_{t}$ and $\mathcal{M}_{t}$ is the moduli space of dHYM connections with phase $e^{\sqrt{-1}\theta_{t}}$ of $(L,h_{t})$ 
with respect to $(\omega_{t},g_{t},J_{t})$. 
Especially, there exists a deformation  $B'\ni t\mapsto \nabla_{t}\in \mathcal{M}_{t}$ of $\nabla_0$ along $(\omega_{t},g_{t},J_{t},h_{t},\theta_{t})$. 
\end{enumerate}
\end{theorem}

Unlike Theorem \ref{Main2}, $\Mm$ itself is a torus, that is, 
every deformation is unobstructed and $\Mm$ is connected. 
We also see that dHYM connections are stable 
under small deformations of the K\"ahler structure and the Hermitian metric on $L$. 

By \cite{McLean}, 
the moduli space $\mathcal{M}_{\mathrm{SL}}$ of special Lagrangian submanifolds  
is a $b^{1}(L)$-dimensional manifold, 
where $b^{1}(L)$ is the first Betti number of $L\in \mathcal{M}_{\mathrm{SL}}$. 
By \cite[Theorem 7.4.19]{Kob}, 
deformations of HYM connections are obstructed in general. 
Hence, Theorem \ref{Main1} indicates that the moduli space $\Mm$ of dHYM connections 
is more similar to that of special Lagrangian submanifolds than to that of HYM connections. 
There is also a statement similar to Theorem \ref{Main1} (2) 
for special Lagrangian submanifolds in \cite[Theorem 3.21]{Marshall}. 
\SkipTocEntry \subsection*{Ideas used in the proof}
The moduli space $\Mm_{G_2}$ is defined 
as a zero set of the so-called deformation map $\Ff_{G_2}$. 
We introduce a new coclosed $G_2$-structure 
from the initial $G_2$-structure and a dDT connection 
and describe the linearization of $\Ff_{G_2}$ in a ``nice way''. 
Using this, 
we show that the deformation of dDT connections is controlled by a subcomplex of the canonical complex 
introduced by Reyes Carri\'on \cite{Reyes}. 
This is one of the novelties of this paper. 

The point to choose a canonical orientation 
is that our new coclosed $G_2$-structures are connected in the space of $G_2$-structures 
to the initial $G_2$-structure by construction. 
Investigating this fact in more detail, we see that a family of Fredholm operators describing infinitesimal deformations is 
homotopic to the trivial family in the space of Fredholm operators. 
Then, the homotopy property of vector bundles implies that the determinant line bundle is trivial, and hence, 
the moduli space has a canonical orientation. 

The moduli space $\Mm$ of dHYM connections is also written as a zero set of a deformation map. 
However, we do not need this viewpoint for the proof of Theorem \ref{Main1} (1). 
We use developed tools in K\"ahler geometry and studies of dHYM metrics here.  
In particular, the uniqueness of dHYM metrics in a fixed holomorphic line bundle proved by Jacob and Yau \cite[Theorem 1.1]{JY} is essential. 
On the other hand,  
we need the viewpoint that $\Mm$ is written as a zero set of a deformation map for the proof of 
Theorem \ref{Main1} (2). 
We describe the linearization map in a ``nice way'' as in the dDT case and show that the restricted linearization map is an isomorphism, 
by which we can apply the implicit function theorem. 
\SkipTocEntry \subsection*{Organization of this paper}
This paper is organized as  follows. 
In Section \ref{canonical-complex}, we develop the general treatment for the canonical complex. 
Section \ref{sec:defofdDT} is devoted to the study of the moduli space of dDT connections after establishing some characteristic properties. 
Then, we show Theorem \ref{Main2}. 
Section \ref{sec:defofdHYM} is about the moduli space of dHYM connections. 
In this paper, we prepare the rich appendix. 
Appendix \ref{AppA} gives basic identities and some decompositions for the spaces of differential forms, based on the Hodge theory. 
Appendix \ref{sec:G2 geometry} is basics of $G_{2}$-geometry. 
In Appendix \ref{sec:new G2 str}, we show that 
the newly induced $G_2$-structure introduced in Section \ref{sec:defofdDT} 
is useful to describe the linearization of the deformation map ``nicely". 
In Appendix \ref{sec:dDT irr decomp}, 
we give another description of the defining equation of the dDT connection 
and prove Proposition \ref{prop:dDT norm}. 
Appendix \ref{notation-list} is the list of notation in this paper. 
\SkipTocEntry \subsection*{Acknowledgments}
The authors would like to thank Naichung Conan Leung for his helpful comments to the idea of this paper when they met at Gakushuin University 
and thank Spiro Karigiannis and Henrique N. S\'a Earp for answering our questions on dDT connections. 
They also would like to thank to Hiroshi Konno for his comments to the previous version, 
which strengthened our main theorems and simplified the proofs. 
The authors thank the anonymous referee for carefully reading the previous version of this paper 
and providing useful comments.

\section{The canonical complex}\label{canonical-complex}
In this section, we develop the general treatment for the canonical complex. 
From Section \ref{sec:defofdDT}, we start the study of deformations of dDT connections. 
The study of the moduli space is reduced to that of a map $\mathcal{F}:\mathcal{A}\to \mathcal{B}$ 
between some (infinite dimensional) linear spaces $\mathcal{A}$ and $\mathcal{B}$ basically consisting of differential forms with some restrictions. Usually, some gauge group $\mathcal{G}$ acts on $\mathcal{A}$ preserving $\mathcal{F}^{-1}(0)$. 
Then, the moduli space $\mathcal{M}$ will be defined as $\mathcal{F}^{-1}(0)/\mathcal{G}$. 
Formally, assume that the tangent space of a $\mathcal{G}$-orbit 
is the image $D_{0}(\mathcal{C})$ of 
some linear map $D_{0}:\mathcal{C}\to \mathcal{A}$. Put $D_{1}:=\delta \mathcal{F}:\mathcal{A}\to\mathcal{B}$, 
the linearization of $\mathcal{F}$. Then, the following complex plays an important role: 
\[
\xymatrix{
0\ar[r] & \mathcal{C} \ar[r]^-{D_{0}} & \mathcal{A}  \ar[r]^-{D_1} & \mathcal{B} \ar[r] & 0 
}.
\]
The proof to say that $\mathcal{M}$ is a smooth manifold is basically done by the implicit function theorem. 
Then, roughly, the surjectivity of $D_{1}$, this is equivalent to the vanishing of the second cohomology $H^{2}:=\mathcal{B}/\mathop{\mathrm{Im}} D_{1}$, ensures the smoothness of $\mathcal{M}$ and the dimension of $\mathcal{M}$ will be that of 
the first cohomology $H^{1}:=\mathop{\mathrm{Ker}} D_{1}/\mathop{\mathrm{Im}} D_{0}$. 
Fortunately, a complex we will meet in this paper can be considered 
as a subcomplex of the so-called \emph{canonical complex} 
introduced in this section. 
Readers who want to see complexes in our practical setting first can consult \eqref{cpxfordDT} in Subsection \ref{sec:infi deform G2}. 

\SkipTocEntry \subsection{Definition of the canonical complex}
The canonical complex was introduced by Salamon \cite[p. 162]{Salamon} for the $G_2$ case, 
and it was defined more generally by Reyes Carri\'on \cite[Section 2]{Reyes} 
as a generalization of the deformation complex of self-dual connections over a 4-manifold. 

Use the notation of Appendix \ref{sec:G2 geometry}. 
Suppose that $X$ is a 7-manifold with a $G_2$-structure $\varphi \in \Om^3$. 
Define a complex by 
\begin{equation}\label{eq:can cpx G2}
\begin{aligned} 
0 \longrightarrow \Om^0 \stackrel{d}
\longrightarrow  \Om^1 \stackrel{D_1} 
\longrightarrow  \Om^2_7 \stackrel{D_2} 
\longrightarrow  \Om^3_1 
\longrightarrow  0 
\end{aligned}
\end{equation}
with 
\[D_{1}(\alpha)=\pi^2_7 (d\alpha),\quad 
D_{2}(\beta)= \pi^3_1 (d\beta). \]
This is called the \emph{canonical complex} (in the case of $G_{2}$). 
By \cite[Lemma 4]{Reyes}, the sequence of the symbols is exact. 
By \cite[Proposition 11]{Reyes} or \cite[Proposition 2]{FU}, 
\eqref{eq:can cpx G2} is elliptic if $d * \varphi \in \Om^5_7$. 

Now, we assume that 
\emph{
$X$ is compact and a $G_2$-structure $\varphi$ 
is coclosed}, that is, $d * \varphi =0$. 
Then, the canonical complex \eqref{eq:can cpx G2} is elliptic. 
For a fixed $C \neq 0$, 
we consider the following complex: 
\begin{equation}\label{cancpxG}
\begin{aligned}
0 \rightarrow \Om^0 
\stackrel{d} \rightarrow \Om^1 
\stackrel{D_1'} \longrightarrow d \Om^5
\rightarrow 0, 
\end{aligned}
\tag{$\#_{G_2}$}
\end{equation}
where 
\[
D_1' (\alpha) = C d \alpha \wedge * \varphi. 
\]
In the following Lemma \ref{diamcomm2}, 
we will see that the complex \eqref{cancpxG} is considered to be a subcomplex of 
the canonical complex \eqref{eq:can cpx G2} by showing that the following diagram commutes: 
\begin{equation} \label{eq:exact seq G2}
\begin{aligned}
\xymatrix{
0\ar[r] & \Om^0 \ar@{=}[d] \ar[r]^{d} & \Om^1 \ar@{=}[d] \ar[r]^{D_1'}
& \Om^6 \ar[r]^d & \Om^7 \ar[r] & 0\\
0\ar[r] & \Om^0 \ar[r]^{d} & \Om^1 \ar[r]^{D_1} & \Om^2_7 \ar[u]_{T} \ar[r]^{D_2} 
& \Om^3_1 \ar[u]_{T} \ar[r] & 0
}
\end{aligned}
\end{equation}
where $T(\gamma) = C \gamma\wedge \ast \varphi$. 
The complex \eqref{cancpxG} is a subcomplex of the top row, 
and the bottom row is the canonical complex \eqref{eq:can cpx G2}. 

\begin{lemma}\label{diamcomm2}
The diagram \eqref{eq:exact seq G2} commutes. 
\end{lemma}
\begin{proof}
Since $\Om^{2}_{14} \wedge * \varphi =0$, it is clear that $D'_1=T\circ D_1$. 
Recall that for any 3-form $\gamma \in \Om^3$, we have 
$\pi^3_1(\gamma) \wedge * \varphi = \gamma \wedge * \varphi$. 
This together with $d * \varphi = 0$ implies that $d\circ T = T\circ D_{2}$. 
\end{proof}

We denote by $\check H^*$ 
the cohomology of the canonical complex \eqref{eq:can cpx G2}. 
Then, we have $\check H^1 = H^1 (\#_{G_2})$ by Lemma \ref{diamcomm2}. 
It is known that the canonical map 
$H^1_{dR} \rightarrow \check H^1 (= H^1 (\#_{G_2}))$ 
is injective by \cite[Lemma 3]{FU}. 
By this injection, we identify $H^1_{dR}$ with a subspace of $\check H^1 (= H^1 (\#_{G_2}))$. 

Since the canonical complex \eqref{eq:can cpx G2} is elliptic now, 
we can define the Laplacian of the complex. 
We denote by $\check \Hh^k$ the space of harmonic $k$-forms of the Laplacian. 
In particular, we have 
\begin{equation} \label{eq:G2-harmonic 1 forms}
\check \Hh^1=
\{\,\alpha \in \Om^1 \mid D_1' (\alpha) = C d \alpha \wedge * \varphi = 0, 
\ \  d * \alpha = 0 \,\}. 
\end{equation}
The formal adjoint of $D_1'$ is given as follows. 

\begin{lemma} \label{lem:D1' G2 adjoint}
The formal adjoint 
$(D_1')^*: \Om^6 \rightarrow \Om^{1}$
of 
$D'_1: \Om^1 \rightarrow \Om^{6}$ is given by 
$(D_1')^* = * \circ D_1' \circ *$. 
\end{lemma}
\begin{proof}
For any $\alpha \in \Om^1$ and $\beta \in \Om^6$, we have 
\[
\begin{aligned}
\left \la  D_1' (\alpha),  \beta  \right \ra_{L^2}
=
\left \la  C d \left( \alpha \wedge * \varphi \right),  \beta  \right \ra_{L^2}
=&
C \int_X \alpha \wedge * \varphi \wedge * d^* \beta\\
=& 
C \left \la \alpha, * \left( * \varphi \wedge d * \beta \right)  \right \ra_{L^2}, 
\end{aligned}
\]
where $\la \cdot, \cdot \ra_{L^2}$ is the $L^2$ inner product and 
we use $* d^* \beta = * * d * \beta = d * \beta$. 
\end{proof}

\begin{corollary} \label{cor:dim 12 cohom}
We have 
\[
\dim H^1(\#_{G_2}) = \dim \check \Hh^1, \quad
\dim H^2(\#_{G_2}) = \dim \check \Hh^1 - b^1. 
\]
where $b^1 = \dim H^1_{dR}$ is the first Betti number of $X$. 
\end{corollary}
\begin{proof}
The first equation follows from \eqref{eq:G2-harmonic 1 forms}. 
We prove the second equation. 
By the ellipticity of the canonical complex, 
the top row of \eqref{eq:exact seq G2} is also elliptic. Hence, we have 
the $L^2$ orthogonal decomposition
\begin{equation} \label{eq:vanish 2nd cohom 1}
\Om^6 =  
\{\, \beta \in \Om^6 \mid d \beta = (D_1')^* \beta = 0 \,\}
\oplus \Im (D_1') \oplus d^{*} \Om^7. 
\end{equation}
By Lemma \ref{lem:D1' G2 adjoint} and \eqref{eq:G2-harmonic 1 forms}, 
we see that 
\begin{equation} \label{eq:vanish 2nd cohom 2}
\{\, \beta \in \Om^6 \mid d \beta = (D_1')^* \beta = 0 \,\} 
= 
* \check \Hh^1.  
\end{equation}
Let $\Hh^1_d$ be the space of harmonic 1-forms. 
By \eqref{eq:G2-harmonic 1 forms}, we have $\Hh^1_d \subset \check \Hh^1$. 
Let $V$ be the $L^2$ orthogonal complement of $\Hh^1_d$ in $\check \Hh^1$.  
Since any elements of $\check \Hh^1$ are coclosed, 
we see that $V \subset d^* \Om^2$. 
This together with \eqref{eq:vanish 2nd cohom 1} and \eqref{eq:vanish 2nd cohom 2} 
implies that 
\[
d \Om^5 = * V \oplus \Im (D_1'). 
\]
Hence, we obtain $H^2(\#_{G_2}) \cong * V$ and the proof is completed. 
\end{proof}

\begin{corollary} \label{cor:vanish 2nd cohom}
The following conditions are equivalent. 
\begin{enumerate}
\item
$H^2(\#_{G_2}) = \{\, 0 \,\}$, 
\item
$\mathop{\mathrm{Im}} (D_1') = d \Om^5$, 
\item
$H^1_{dR} = \check H^1 (=  H^1 (\#_{G_2}))$. 
\end{enumerate}
\end{corollary}
\begin{proof}
By the definition of $H^2(\#_{G_2})$, 
(1) and (2) are equivalent. 
The equivalence of (1) and (3) follows from the second equation of Corollary \ref{cor:dim 12 cohom}. 
\end{proof}
\section{Deformations of dDT connections for $G_2$-manifolds}\label{sec:defofdDT}
Let $X$ be a 7-manifold with a $G_2$-structure $\varphi \in \Om^3$ 
and $L \to X$ be a smooth complex line bundle with a Hermitian metric $h$.

\begin{definition}
A Hermitian connection $\nabla$ of $(L,h)$ satisfying 
\begin{equation}\label{defofdDT2}
\frac{1}{6} F_\nabla^3 + F_\nabla \wedge \ast\varphi=0
\end{equation}
is called a \emph{deformed Donaldson--Thomas (dDT) connection}. 
Here, $F_\n$ is the curvature 2-form of $\n$. We consider it as a $\i \R$-valued closed 2-form on $X$. 
\end{definition}
\subsection{Some properties of dDT connections}
Use the notation (and identities) of Appendix \ref{sec:G2 geometry}. 
Before studying deformations of dDT connections, 
we give some properties of dDT connections. 

\begin{lemma}
A Hermitian connection $\nabla$ of $(L,h)$ 
satisfies \eqref{defofdDT2} if and only if 
\begin{equation} \label{lem:equiv dDT G2 1}
* F_\nabla + \varphi \wedge F_\nabla
=
- \frac{1}{6} (* F_\nabla^3) \wedge * \varphi. 
\end{equation}
\end{lemma}
\begin{proof}
Set 
$(F_\nabla)_7 = \pi^2_7(F_\n) \in \i \Om^2_7$. 
Then, $\pi^2_7(F_\n)=\i i(u) \varphi$ 
for some vector field $u$. 
Then, we have 
\[
* F_\nabla + \varphi \wedge F_\nabla = 3* (F_\nabla)_7 = 3 \i u^\flat \wedge * \varphi. 
\]
Since $\bullet \wedge * \varphi: \Om^1 \rightarrow \Om^5$ is injective, 
we see that \eqref{lem:equiv dDT G2 1} is equivalent to 
$3 \i u^\flat = - * F_\nabla^3/6$. 
On the other hand, since 
$F_\nabla \wedge * \varphi 
= (F_\nabla)_7 \wedge * \varphi = 3 \i  * u^\flat$, 
we see that \eqref{defofdDT2} is equivalent to \eqref{lem:equiv dDT G2 1}. 
\end{proof}

\begin{remark}
If a Hermitian connection $\nabla$ of $(L,h)$ is a dDT connection, that is, 
$F_\nabla^3/6 + F_\nabla \wedge \ast\varphi=0$, 
we have $\varphi \wedge * F_\n^2 =0$. 
It is because we can apply Corollary \ref{cor:1+F 00} pointwisely since $-\i F_\n \in \Om^2$.  
As explained in \cite{KY1}, this property is also implied by the real Fourier--Mukai transform. 
\end{remark}

Set $(F_\nabla)_j = \pi^2_j (F_\n) \in \i \Om^2_j$ for $j=7,14$. 
If $\n$ is a dDT connection, 
the norms of $(F_\nabla)_7$ and $(F_\nabla)_{14}$ satisfy the following relation. 
Though we do not use Proposition \ref{prop:dDT norm} in this paper, 
we believe that this decomposition will be helpful for the further research of dDT connections. 
The proof is given in Appendix \ref{sec:dDT irr decomp}.

\begin{proposition} \label{prop:dDT norm}
Suppose that $\n$ is a dDT connection. Then, 
\begin{enumerate}
\item
if $(F_\nabla)_{14}=0$, we have $(F_\nabla)_7=0$ or $|(F_\nabla)_7| = 3$. 
\item
We have 
\[
|(F_\nabla)_7| \leq \sqrt{2 |(F_\nabla)_{14}|^2 + 12} 
\cos \left( \frac{1}{3} \arccos \left( \frac{|(F_\nabla)_{14}|^3}{(|(F_\nabla)_{14}|^2+6)^{3/2}} \right) \right).
\]
\end{enumerate}
\end{proposition}

Next, we explain the relation between dDT connections and dHYM connections. 
Let $(Y^6,\om,g,J,\Om)$ be a real 6-dimensional Calabi--Yau manifold and 
$L \rightarrow Y$ be a smooth complex line bundle with a Hermitian metric $h$. 
Then, $X^7:=S^1 \times Y^6$ has an induced $G_{2}$-structure $\varphi$ given by 
\[\varphi:=dx \wedge \om + \mathop{\mathrm{Re}} \Om, \]
where $x$ is a coordinate of $S^1$. 
Lee and Leung \cite[Section 2.4.1]{LL} proved that 
if a Hermitian connection $\nabla$ of $(L,h)$ 
is dHYM with phase $1$ on $Y^{6}$, 
the pullback of $\nabla$ is a dDT connection on $X^7$. 
The converse also holds. 

\begin{lemma}
A Hermitian connection $\nabla$ of $(L,h)$ 
is a dHYM connection with phase $1$ 
if and only if 
the pullback of $\nabla$ is a dDT connection. 
\end{lemma}
\begin{proof}
Since $\ast\varphi= \om^2/2-dx \wedge \Im \Om$, 
we have 
\[
\frac{1}{6} F_\nabla^3 + F_\nabla \wedge * \varphi
=
\frac{1}{6} \left( F_\nabla^3+3 F_\nabla \wedge \om^2 \right)
- dx \wedge F_\nabla \wedge \Im \Om.
\]
Since we know that  
\[
0=\i \Im (\om+F_\nabla)^3= 3 \om^2 \wedge F_\nabla + F_\nabla^3, 
\]
$F_\nabla^3/6 + F_\nabla \wedge \ast \varphi=0$ if and only if 
$\Im (\om+F_\nabla)^3=0$ and $F_\nabla \wedge \Im \Om=0$. 
We can show that $F_\nabla \wedge \Im \Om=0$ if and only if $F_\nabla^{0,2}=0$ as follows. 
Since $F_\nabla$ is $\i \R$-valued, it is immediate to see that 
$F_\nabla \wedge \Im \Om = 0$ if $F_\nabla^{0,2}=0$. 
Conversely, suppose that $F_\nabla \wedge \Im \Om = 0$. 
Since $F_\n$ is $\i \R$-valued 
and $F_{\nabla}^{2,0}\wedge \Omega= F_{\nabla}^{1,1}\wedge \Omega=0$, 
we see that 
\[
\begin{aligned}
F_\nabla \wedge \Im \Om
=&\i \Im \left(\frac{F_\n}{\i} \wedge \Om \right)\\
=&\i \Im \left(\frac{F_\n^{0,2}}{\i} \wedge \Om \right)
= \frac{1}{2 \i} (F_\n^{0,2} \wedge \Om + \overline{ F_\n^{0,2} \wedge  \Om}). 
\end{aligned}
\]
Since the first term is in $\Om^{3,2}(Y^6)$ and the second term is in $\Om^{2,3}(Y^6)$, 
it follows that $F_\n^{0,2} \wedge \Om = 0.$
Since $\bullet \wedge \Om: \Om^{0,2}(Y^6) \rightarrow \Om^{3,2}(Y^6)$ 
is injective, which follows from the pointwise calculation, 
we obtain $F_\n^{0,2}=0$. 
\end{proof}
\subsection{Definition of the moduli space of dDT connections} \label{sec:def dDT}
In this subsection, we define the moduli space of dDT connections. 
Let $X$ be a 7-manifold with a $G_2$-structure $\varphi \in \Om^3$ 
and $L \to X$ be a smooth complex line bundle with a Hermitian metric $h$. 
Set 
\[
\begin{aligned}
\mathcal{A}_{0}=&\{\,\ \mbox{Hermitian connections of }(L,h) \,\}\\
=&\nabla + \i \Om^1 \cdot \id_L, 
\end{aligned}
 \]
where $\nabla \in \Aa_{0}$ is any fixed connection.  
Define a deformation map $\Ff_{G_2}$ by 
\begin{equation} \label{eq:deform map G2}
\Ff_{G_2}: \Aa_{0} \rightarrow \i \Om^{6}, \quad 
\nabla \mapsto \frac{1}{6} F_\nabla^3 + F_\nabla \wedge * \varphi. 
\end{equation} 
Each element of $\Ff_{G_2}^{-1}(0)$ is a dDT connection. 

Let $\Gg_U$ be the group of unitary gauge transformations of $(L,h)$. 
Precisely, 
\[
\Gg_U= \{\, f \cdot \id_L \mid f \in \Om^0_{\C}, \ |f|=1 \,\} \cong C^\infty(X, S^1).
\]
The action $\Gg_U \times \Aa_0 \rightarrow \Aa_0$ is defined by 
$(\lambda, \nabla) \mapsto \lambda^{-1} \circ \nabla \circ \lambda$.
When $\lambda=f \cdot \id_L$, we have 
\[
\lambda^{-1} \circ \nabla \circ \lambda = \nabla + f^{-1}df \cdot \id_L. 
\]
Thus, the $\Gg_U$-orbit through $\n \in \Aa_0$ is given by $\n + \Kk_U \cdot \id_L$, where 
\begin{equation} \label{eq:Gu orbit G2}
\Kk_U:= \{\, f^{-1} d f \in \i \Om^1 \mid f \in \Omega^{0}_{\C}, \ |f|=1 \,\}. 
\end{equation}

Since the curvature 2-form $F_\nabla$ is invariant under the action of $\Gg_U$, 
$\Ff_{G_2}$ reduces to 
\[
\begin{aligned}
\underline{\Ff_{G_2}}: \Aa_{0}/\Gg_U & \rightarrow \i \Om^{6}, \quad
[\nabla] \mapsto \frac{1}{6} F_\nabla^3 + F_\nabla \wedge * \varphi. 
\end{aligned}
\]
\begin{definition}
The {\em moduli space} of deformed Donaldson--Thomas connections of $(L,h)$, 
denoted by $\Mm_{G_2}$, 
is given by 
\[
\Mm_{G_2}= \underline{\Ff_{G_2}}^{-1}(0) = \Ff_{G_2}^{-1}(0)/\Gg_U. 
\]
\end{definition}

\begin{remark} \label{rem:metric B0 G2}
Since $\Aa_0$ is an affine Fr\'echet space, 
we have an induced metric $d_{\Aa_0}$ on $\Aa_0$. 
Explicitly, we have 
$$
d_{\Aa_0} (\n, \n') 
= \sum_{n=0}^\infty \frac{1}{2^n} \cdot \frac{|\n - \n'|_{C^n}}{1+ |\n - \n'|_{C^n}}, 
$$
where $|\cdot|_{C^n}$ is the $C^n$ norm on $\i \Om^1 \cdot \id_L \cong \Om^1$. 
As in \cite[(4.2.3)]{DK}, we can define a metric
on $\Bb_0 = \Aa_0 /\Gg_U$ by 
$$
d_{\Bb_0} ([\n],[\n']) 
= \inf_{\lambda \in \Gg_U} d_{\Aa_0} (\n, \lambda^{-1} \circ \nabla' \circ \lambda). 
$$
Moreover, 
the $C^\infty$ (quotient) topology on $\Bb_0$ 
agrees with the topology induced by $d_{\Bb_0}$. 
In particular, $\Bb_0$ (with the $C^\infty$ topology) is paracompact and Hausdorff. 
Similarly, $\Mm_{G_2}$ (with the induced topology from $\Bb_0$)
is also metrizable, and hence, it is paracompact and Hausdorff. 
\end{remark}
\subsection{The infinitesimal deformation} \label{sec:infi deform G2}
For the rest of Section \ref{sec:defofdDT}, we suppose that 
$X^7$ is compact, the $G_2$-structure $\varphi$ is coclosed, 
that is, $d * \varphi =0$,  
and 
$\Mm_{G_2} \neq \emptyset$. 
In this subsection, we study the infinitesimal deformation of dDT connections. 

To study the infinitesimal deformation is highly nontrivial. 
The key is Theorem \ref{thm:1+F}, which describes 
$F_\n^2/2 + *\varphi$ in terms of a new coclosed $G_2$-structure 
defined by $\varphi$ and $\n$. 
This enables us to describe the linearization of $\Ff_{G_2}$ ``nicely''. 
Then, we can control the deformations of dDT connections by 
the complex \eqref{cpxfordDT} in this subsection. 

We now describe the linearization of $\Ff_{G_2}$. 
Use the notation of Appendix \ref{sec:new G2 str}. 
Fix $\nabla \in \Ff_{G_2}^{-1}(0)$. 
Since $F_\n$ is $\i \R$-valued, 
$\nabla \in \Ff_{G_2}^{-1}(0)$ implies that 
$- \i F_\n \in \Om^2$ satisfies 
$- (- \i F_\n)^3/6 + (- \i F_\n) \wedge * \varphi = 0$. 
Then, by Theorem \ref{thm:1+F}, 
we have $1 + \la F_\n^2, * \varphi \ra/2 \neq 0$. 
Define a new $G_2$-structure $\tilde \varphi_\n$ by 
\begin{equation} \label{eq:newG2str}
\tilde \varphi_\n =  \left | 1 + \frac{1}{2} \la F_\n^2, * \varphi \ra \right |^{-3/4}  
(\id_{TX} + (- \i F_\n)^\sharp)^* \varphi. 
\end{equation}
Note that 
$\tilde \varphi_\n$ is coclosed by \eqref{eq:1+F conf Hodge} since $* \varphi$ and $F_\n$ are closed. 
Denote by $\tilde{g}_{\nabla}$ a Riemannian metric on $X$ induced from $\tilde \varphi_\n$ via \eqref{eq:form-1def}. 
Denote by $\tilde *_\n$ and $d^{\tilde *_\n}$ 
the Hodge star and the formal adjoint of the exterior derivative $d$ with respect to $\tilde g_{\n}$, respectively. 

\begin{remark}
The operator $\id_{TX} + (- \i F_\n)^\sharp$ 
also appears in the graphical setting. 
For a smooth map $f:B \to T^n$, 
define the inclusion $\iota: B \to B \times T^n$ by $\iota (x)=(x, f(x))$. 
We can define the real Fourier--Mukai transform $\n$ of the graph of $f$ (the image of $\iota$) 
as in \cite[Section 2]{KY1}. 
Then, by the proof of \cite[Theorem 5.7]{KY1}, we see that 
the derivative of $\iota$ is given by $(\id_{TX} + (- \i F_\n)^\sharp)|_{TB}$. 
In particular, 
the tangent space of the graph of $f$ in $B \times T^n$ is given by 
$(\id_{TX} + (- \i F_\n)^\sharp)(TB)$ and 
$((\id_{TX} + (- \i F_\n)^\sharp)|_{TB})^* \varphi$ is the pullback of $\varphi$ by $\iota$. 

Since dDT connections are introduced by the real Fourier--Mukai transform of 
graphical associative submanifolds in \cite{LL}, 
we can expect that the operator $\id_{TX} + (- \i F_\n)^\sharp$ will play a role.  
We can show that this is indeed the case. 
\end{remark}

Then, the linearization of $\Ff_{G_2}$ at $\nabla$ is given as follows. 

\begin{proposition}
The linearization 
$\delta_\nabla \Ff_{G_2}: \i \Om^1 \rightarrow \i \Om^{6}$ 
of $\Ff_{G_2}$ at $\nabla \in \Ff_{G_2}^{-1}(0)$ is given by  
\begin{equation}\label{eq:1stder G2}
(\delta_\nabla \Ff_{G_2}) (\i b) 
= C \i d b \wedge  \tilde *_\n \tilde \varphi_\n, 
\end{equation}
where $C=1$ if $1+ \la F_\n^2, * \varphi \ra/2>0$ and $C=-1$ if it is negative. 
\end{proposition}

\begin{proof}
By the definition of $\Ff_{G_2}$ in \eqref{eq:deform map G2}, 
$\delta_\n \Ff_{G_2}$ is given by 
\[
\begin{aligned}
\delta_\n \Ff_{G_2} (\i b) 
= & 
\i d b \wedge \left(\frac{1}{2} F_\nabla^2 + * \varphi \right)\\
= & 
\i d b \wedge \left(- \frac{1}{2} (- \i F_\n)^2 + * \varphi \right) 
\end{aligned}
\]
for $b \in \Om^1$. 
Since $- (- \i F_\n)^3/6 + (- \i F_\n) \wedge * \varphi =0$, 
\eqref{eq:1+F conf Hodge} implies \eqref{eq:1stder G2}. 
\end{proof}

For a dDT connection $\n$, $1+ \la F_\n^2, * \varphi \ra/2$ is nowhere vanishing, 
and hence, has constant sign if $X$ is connected. 
Then, we see the following. 

\begin{corollary}
When $X$ is connected and the $G_2$-structure $\varphi$ is torsion-free, 
the sign of $1+ \la F_\n^2, * \varphi \ra/2$ 
is determined topologically. That is, the sign of 
$$
\int_X \left( 1+ \frac{1}{2} \la F_\n^2, * \varphi \ra \right) \vol 
= 
{\rm Vol}(X) + (-2 \pi^2 c_1(L)^2 \cup [\varphi]) \cdot [X], 
$$
where $c_1(L)$ is the first Chern class of $L$, 
agrees with that of $1+ \la F_\n^2, * \varphi \ra/2$. 

In particular, if ${\rm Vol}(X) + (-2 \pi^2 c_1(L)^2 \cup [\varphi]) \cdot [X]$ vanishes, 
we see that there are no dDT connections for $L \to X$. 
\end{corollary}

For any fixed $\nabla \in \Ff_{G_2}^{-1}(0)$, 
consider the following complex 
\begin{equation}\label{cpxfordDT}
\begin{aligned}
0 \rightarrow \i \Om^0 
\stackrel{d} \rightarrow \i \Om^1 
\stackrel{d (\tilde *_\n \tilde \varphi_\n \wedge \bullet)} 
\longrightarrow \i d \Om^5
\rightarrow 0, 
\end{aligned}
\tag{$\#_\n$} 
\end{equation}
which is the complex \eqref{cancpxG} in Section \ref{canonical-complex} 
for $\tilde \varphi_\n$ multiplied by $\i$. 
In particular, \eqref{cpxfordDT} is considered to be a subcomplex of the canonical complex, which is elliptic. 

By \eqref{eq:Gu orbit G2}, the tangent space of $\Gg_U$-orbit through $\n$ is 
identified with $d \Om^1$. 
Hence, 
by \eqref{eq:1stder G2}, 
the first cohomology $H^1(\#_\n)$ of \eqref{cpxfordDT} is considered to be the tangent space of $\Mm_{G_2}$. 
We show that 
the second cohomology $H^2(\#_\n)$ is the obstruction space in Theorem \ref{thm:moduli MG2}. 

By Corollary \ref{cor:dim 12 cohom}, the expected dimension of $\Mm_{G_2}$ is given as follows. 
\begin{lemma} \label{lem:expected dim G2}
The expected dimension of $\Mm_{G_2}$ is given by 
\[
\dim H^1 (\#_\n) - \dim H^2 (\#_\n) 
\]
which is equal to $b^{1}$, the first Betti number of $X$. 
\end{lemma}
\subsection{The local structure of $\Mm_{G_2}$}\label{sec:local str MG2}
Before studying the global structure of $\Mm_{G_2}$, we study its local structure, that is, we consider the smoothness of $\Mm_{G_2}$. 
First, note the following. 

\begin{lemma} \label{onimage G2}
The image of $\Ff_{G_2}$ is contained in $\i d \Om^{5}$. 
\end{lemma}

\begin{proof}
Given two connections $\nabla, \nabla' \in \Aa_{0}$, 
we know that $F_{\nabla'}-F_\nabla \in \i d \Om^1$. 
Then, it follows that 
\[
\Ff_{G_2} (\nabla') - \Ff_{G_2}(\nabla) \in \i d \Om^5. 
\]
Thus, if we take $\nabla \in \Ff_{G_2}^{-1}(0)$, the statement follows. 
\end{proof}

It is standard to consider a 
slice to the action of $\Gg_U$ 
to study the local structure of $\Mm_{G_2}$. 
First, we prove the following. 

\begin{lemma} \label{lem:slice}
We have $\i d \Om^0 \subset \Kk_U$, where $\Kk_U$ is defined by \eqref{eq:Gu orbit G2}. 
Conversely, for $k \in \N$ and $0<\alpha<1$, 
the elements of $\Kk_U$ with the small $C^{k, \alpha}$ norms
are contained in $\i d \Om^0$. 
\end{lemma}

\begin{proof}
For any function $f_0 \in \Om^0$, define $f= e^{\i f_0} \in \Om^0_{\C}.$
Then, $|f|=1$ and $f^{-1} df = \i d f_0$, which implies $\i d \Om^0 \subset \Kk_U$. 
Let $\delta_0 \iota: \i d \Om^0 \rightarrow T_0 \Kk_U$ be the differential of 
the inclusion map $\iota: \i d \Om^0 \hookrightarrow \Kk_U$ at $0$. 
Since $T_0 \Kk_U = \i d \Om^0$, $\delta_0 \iota$ is actually the identity map.  

Let $\Kk_U^{k, \alpha}$ be the completion of $\Kk_U$ with respect to the $C^{k,\alpha}$ norm. 
Then, $\iota: \i d \Om^0 \hookrightarrow \Kk_U$ extends to 
$\iota^{k,\alpha}: \i d C^{k+1, \alpha}(X, \Lambda^0T^*X) \rightarrow 
\Kk^{k, \alpha}_U$ and $\delta_0 \iota^{k, \alpha}$ is the identity map. 
Then, by the inverse function theorem, $\iota^{k, \alpha}$ is a local homeomorphism 
and the proof is done. 
\end{proof}

Fix any $\nabla \in \Ff_{G_2}^{-1}(0)$. 
Set 
\[
\Om^1_{d^{\tilde *_\n}} 
= \{\, d^{\tilde *_\n} \mbox{-closed 1-forms on } X \,\}.
\]

\begin{proposition}\label{p3.12}
The map $p: \i \Om^1_{d^{\tilde *_\n}} \rightarrow \Aa_0 /\Gg_U$ defined by 
\[A \mapsto [\nabla+A \cdot \id_L]\]
gives a homeomorphism from a neighborhood of $0 \in \i \Om^1_{d^{\tilde *_\n}}$ 
to that of $[\nabla] \in \Aa_0/\Gg_U$. 
\end{proposition}

\begin{proof}
By \eqref{eq:Gu orbit G2}, 
\[
\Aa_0 /\Gg_U
= \nabla + \i \Om^1/\Kk_U \cdot \id_L. 
\]
Since $\i d \Om^0 \subset \Kk_U$ by Lemma \ref{lem:slice}, 
we see that $p$ is surjective. 

Next, we show that $p$ is injective around $0$. 
Suppose that $p(A_1)=p(A_2)$ for $A_1,A_2 \in \i \Om^1_{d^{\tilde *_\n}}$. 
Then, we have  
$A_1-A_2 \in \Kk_U$. 
If $A_1$ and  $A_2$ are sufficiently close to $0$ with respect to $C^\infty$ topology, 
they are sufficiently close to $0$ with respect to $C^{k, \alpha}$ topology 
for fixed $k \in \N$ and $\alpha \in (0,1)$. 
Then, by Lemma \ref{lem:slice}, $A_1-A_2$ is exact, 
which implies that 
$A_1-A_2 \in \i \left( \Om^1_{d^{\tilde *_\n}}\cap d \Om^0 \right) = \{\, 0 \,\}$.
\end{proof}

By \eqref{eq:Gu orbit G2}, the tangent space of $\Gg_U$-orbit through $\n$ is 
identified with $d \Om^1$. 
Hence, 
a neighborhood of $[\nabla]$ in $\Mm_{G_2}$ is homeomorphic to 
that of $0$ in 
\[
\tilde \Ss_\n = \{\, a \in \Om^1_{d^{\tilde *_\n}} \mid \Ff_{G_2} (\nabla + \i a \cdot \id_L) = 0 \,\}. 
\]

\begin{theorem} \label{thm:moduli MG2}
If $H^2(\#_\nabla) = \{\, 0 \,\}$ for $\nabla \in \Ff^{-1}_{G_2}(0)$, 
the moduli space $\Mm_{G_2}$ is a smooth manifold near $[\nabla]$ of dimension $b^1$, 
where $b^1$ is the first Betti number of $X$. 
\end{theorem}

\begin{proof}
We only have to show that $\tilde \Ss_\n$ is a smooth manifold near $0$. 
Since $\mathop{\mathrm{Im}} (\Ff_{G_2}) \subset \i d \Om^5$ by Lemma \ref{onimage G2} 
and $\mathop{\mathrm{Im}} (\delta_\n \Ff_{G_2}) = \i d \Om^5$ 
by $H^2(\#_\nabla) = \{\, 0 \,\}$, 
we can apply the implicit function theorem 
(after the Banach completion) to 
\[
\Om^1_{d^{\tilde *_\n}} \rightarrow \i d \Om^5, \quad 
a \mapsto \Ff_{G_2} (\nabla + \i a \cdot \id_L). 
\]
Then, we see that 
$\tilde \Ss_\n$ is a smooth manifold near $0$. 
Its dimension is $\dim H^1 (\#_\n)$, 
which is equal to $b^1$ by Corollary \ref{cor:vanish 2nd cohom} 
(or Lemma \ref{lem:expected dim G2}). 

As for the regularity of elements in $\Ff_{G_2}^{-1}(0)$ in the Banach completion, 
note that 
\[
\tilde \Ss_\n = 
\{\, a \in \Om^1 \mid \Ff_{G_2} (\nabla + \i a \cdot \id_L) = 0, \ d^{\tilde *_\n} a =0 \,\}. 
\]
Since 
\[
(\delta_\nabla \Ff_{G_2}, d^{\tilde *_\n}): \i \Om^1
\rightarrow \i \left( \Om^6 \oplus \Om^0 \right) 
\]
is overdetermined elliptic by 
the ellipticity of \eqref{cpxfordDT} at $\i \Om^1$, 
we see that $\tilde \Ss_\n$ is the solution space of 
an overdetermined elliptic equation around $0$. 
Thus, all solutions around $0$ are smooth. 
\end{proof}

\begin{remark}
If $\tilde \varphi_\n$ is torsion-free or nearly parallel, 
$\Mm_{G_2}$ is smooth near $[\n]$. 
It is because by \cite[Theorems 9 and 10]{FU}, 
we have $H^1_{dR} = H^1 (\#_\n)$, which is equivalent to  
$H^2 (\#_\n) = \{\, 0 \,\}$ 
by Corollary \ref{cor:vanish 2nd cohom}. 

In particular, if $\varphi$ is torsion-free or nearly parallel and 
a connection $\n$ is flat, which is obviously a dDT connection, 
we have $\tilde \varphi_\n = \varphi$, 
and hence, $\Mm_{G_2}$ is smooth near $[\n]$. 
\end{remark}

\subsection{Connected components of $\Mm_{G_2}$}
In the previous subsection, we have seen that the moduli space $\Mm_{G_2}$ has a $b^1$-dimensional manifold structure 
around $[\nabla]$ with $H^2 (\#_\n) = \{\, 0 \,\}$. 
In this subsection, we show that a connected component containing such a $[\nabla]$ is a torus.

The key is the following fact. 
For a dDT connection $\nabla$, a Hermitian connection of the form $\n + \i a\cdot \id_L$ with $a\in Z^1$ is also a dDT connection, 
where $Z^1$ is the space of closed 1-forms on $X$. 
Since the tangent space of $\Gg_U$-orbit through $\n$ is 
identified with $d \Om^0$ by \eqref{eq:Gu orbit G2}, 
roughly we can expect that the set $\{\,[\nabla+\i a\cdot \id_L]\mid a\in Z^1 \,\}$ gives a $b^1$-dimensional family in $\Mm_{G_2}$. 
Considering the dimensions, we can conclude that the set coincides with a connected component of $\Mm_{G_2}$. 
This expectation is justified as the following theorem and this is one of main theorems in this paper. 

\begin{theorem}\label{newtheorem}
Each connected component of $\Mm_{G_2}$ including an element $[\n]$ such that $H^2(\#_\nabla) = \{\, 0 \,\}$ is 
homeomorphic to a $b^1$-dimensional torus. 
\end{theorem}
\begin{proof}
Take $[\n]\in \Mm_{G_2}$ such that $H^2(\#_\nabla) = \{\, 0 \,\}$. 
Set $\mathcal{A}_{\n}:=\nabla + \i Z^1 \cdot \id_L$. 
It is clear that $\mathcal{A}_{\n}/\Gg_U$ is included in $\Mm_{G_2}$. 
From \cite[Lemma 4.1]{KY2}, 
it follows that $\mathcal{A}_{\n}/\Gg_U$, which can be identified with $(\sqrt{-1}Z^{1})/\Kk_{U}$, is 
homeomorphic to a $b^1$-dimensional torus $H^1(X,\mathbb{R})/2\pi H^1(X,\mathbb{Z})$, especially $\mathcal{A}_{\n}/\Gg_U$ is compact. 
Then, since $\Mm_{G_2}$ is Hausdorff as noted in Remark \ref{rem:metric B0 G2} and $\mathcal{A}_{\n}/\Gg_U$ is a compact subset, 
$\mathcal{A}_{\n}/\Gg_U$ is a closed subset in $\Mm_{G_2}$. 

Next, we show that $\Aa_\n/\Gg_U$ is also open in $\Mm_{G_2}$. 
Fix $[\n']=[\n+\sqrt{-1}a\cdot\mathrm{id}_{L}]$ in $\mathcal{A}_{\n}/\Gg_U$ with $a\in Z^{1}$. 
Then, since $F_{\n'}=F_{\n}$, we have $\tilde \varphi_{\n'}=\tilde \varphi_{\n}$, see the definition of $\tilde  \varphi_{\n}$ in \eqref{eq:newG2str}. 
This implies that $H^2(\#_{\n'})=H^2(\#_\n) = \{\, 0 \,\}$. 
Thus, we can apply Theorem \ref{thm:moduli MG2} for $\n'$ 
and say that $\Mm_{G_2}$ is a smooth manifold near $[\n']$ of dimension $b^1$. 
On the other hand, $\mathcal{A}_{\n}/\Gg_U$ is also a smooth manifold of dimension $b^1$. 
These two facts imply that two spaces coincide on a small neighborhood of $[\n']$. 
This means that $\mathcal{A}_{\n}/\Gg_U$ is open in $\Mm_{G_2}$. 
Since it is also closed, we conclude that the connected component including $[\n]$ agrees with 
$\mathcal{A}_{\n}/\Gg_U$, which is homeomorphic to a $b^1$-dimensional torus. 
\end{proof}
\subsection{Varying the $G_2$-structure}\label{202106181615}
Let $X^7$ be a compact 7-manifold with a coclosed $G_2$-structure $\varphi$ 
and $L \to X$ be a smooth complex line bundle with a Hermitian metric $h$. 
Set $\psi = * \varphi$. 
In Theorems \ref{thm:moduli MG2} and \ref{newtheorem}, 
we gave the condition that a connected component of the moduli space of dDT connections 
is smooth and homeomorphic to a $b^1$-dimensional torus. 
In this subsection, 
we show that it is satisfied if we perturb $\psi$ generically 
in the same cohomology class as $\psi$ in some cases. 

We also call a 4-form that is pointwisely identified with \eqref{varphi*} 
a $G_2$-structure. 
Define $\Pp_\psi \subset \psi + d \Om^3$ by 
\[
\Pp_\psi = \{\,\psi' \in \Om^4 \mid \psi' \mbox{ is a $G_2$-structure, }d \psi'=0 \mbox{ and }
[\psi'] = [\psi] \in H^4_{dR} \,\}. 
\]
Define a map 
$\wFf : \Pp_\psi \times \Aa_0 \to \i \Omega^6$  
by 
\[\wFf (\psi', \n') = \frac{1}{6} F_{\n'}^3 + F_{\n'} \wedge \psi'
\]
and set 
\[
\widehat \Mm_{G_2} = \left \{\,(\psi', [\n']) \in  \Pp_\psi \times \Aa_0 /\Gg_U \mid \wFf (\psi', \n') = 0\,\right\}. 
\]
For each $\psi' \in \Pp_\psi$, set 
$\mathcal{M}_{G_2, \psi'} = \widehat \Mm_{G_2} \cap (\{\, \psi' \,\} \times \Aa_0/\Gg_U)$ and this 
is the moduli space of deformed Donaldson--Thomas connections 
with respect to the $G_2$-structure $\psi'$. 

For the rest of this subsection, we suppose that $\mathcal{M}_{G_2, \psi} \neq \emptyset$ 
for the initial coclosed $G_2$-structure $\psi$. 
As in Lemma \ref{onimage G2}, we see the following. 

\begin{lemma} \label{onimae G2 para}
The image of $\wFf$ is contained in $\i d \Om^5$. 
\end{lemma}

\begin{proof}
For any $\psi_1, \psi_2 \in \Pp_\psi$ and $\nabla_1, \nabla_2 \in \Aa_{0}$, 
we know that $\psi_2 - \psi_1 \in d \Om^3$ and $F_{\nabla_2}-F_{\nabla_1} \in \i d \Om^1$. 
Then, 
$$
F_{\n_2} \wedge \psi_2 - F_{\n_1} \wedge \psi_1
= (F_{\n_2} - F_{\n_1} ) \wedge \psi_2 + F_{\n_1} \wedge (\psi_2 - \psi_1) \in \i d \Om^5, 
$$
and hence,  
\[
\wFf (\psi_2, \n_2) - \wFf (\psi_1, \n_1) \in \i d \Om^5. 
\]
Thus, if we take $(\psi_1, \nabla_1) \in \wFf^{-1}(0)$, the statement follows. 
\end{proof}

Next, we prove the following to prove the surjectivity of the linearization of $\wFf$. 
The similar statements are given by Gayet \cite{Gayet} in the case of associative submanifolds 
and by Mu\~noz and Shahbazi \cite{MS} in the case of $\Sp$-instantons. 

\begin{lemma} \label{lem:L2 oncomp}
Use the notation of Subsection \ref{sec:infi deform G2}. 
Let $\n$ be a dDT connection with respect to $\psi = * \varphi$. 
Denote by  
$( \Im \delta_{(\psi, \n)} \wFf )^\perp$ 
the space of elements in $\i d \Om^5$ orthogonal to 
the image of the linearization $\delta_{(\psi, \n)} \wFf$ at $(\psi, \n)$ 
with respect to the $L^2$ inner product induced from $\tilde g_{\n}$.  
\begin{enumerate}
\item 
The map $\delta_{(\psi, \n)} \wFf: d \Om^3 \times \i \Om^1 \rightarrow \i d \Om^5$ is surjective 
if and only if 
$( \Im \delta_{(\psi, \n)} \wFf)^\perp = \{\, 0 \,\}$. 

\item
For $\eta \in d^{\tilde *_\n} \Om^2$, 
$\tilde *_\n \eta \in ( \Im \delta_{(\psi, \n)} \wFf )^\perp$ 
if and only if 
\[
F_\n \wedge d \eta =0 \quad \mbox{and} \quad \psi \wedge d \eta = 0. 
\]
\end{enumerate}
\end{lemma}

\begin{proof}
We first prove (1). 
By the definition of $\wFf$, we see that 
\begin{align} \label{eq:L2 oncomp 1}
\delta_{(\psi, \n)} \wFf (d \xi, \i b) 
= 
F_\n \wedge d \xi + \i db \wedge \left( \frac{1}{2} F_\n^2 +\psi \right) 
\end{align}
for $\xi \in \Om^3$ and $b \in \Om^1$. 
For simplicity, define 
$D_1: \i \Om^1\rightarrow \i d \Om^5$ and $D_2: d \Om^3 \rightarrow \i d \Om^5$
by 
\[
D_1 (\i b) = \i db \wedge \left( \frac{1}{2} F_\n^2 +\psi \right), \quad 
D_2 (d \xi) = F_\n \wedge d \xi. 
\]
Then, $\Im \delta_{(\psi, \n)} \wFf = \Im D_1 + \Im D_2$. 
By the proof of Corollary \ref{cor:dim 12 cohom} and \eqref{eq:1stder G2}, 
there exists a finite dimensional subspace $W \subset \i d \Om^5$ 
such that 
\[
\i d \Om^5 = W \oplus \Im D_1, 
\]
which is an orthogonal decomposition 
with respect to the $L^2$ inner product induced from $\tilde g_{\n}$. 
Let $U$ be the orthogonal complement of $W \cap (\Im D_1 + \Im D_2)$ in $W$ 
with respect to the $L^2$ inner product induced from $\tilde g_{\n}$. 
Then, 
since $\Im D_1 + \Im D_2 = (W \cap (\Im D_1 + \Im D_2)) \oplus \Im D_1$, 
we obtain the following orthogonal decomposition
\[
\i d \Om^5 = U \oplus (\Im D_1 + \Im D_2) 
\]
with respect to the $L^2$ inner product induced from $\tilde g_{\n}$. 
By construction, 
$U = ( \Im \delta_{(\psi, \n)} \wFf )^\perp$, 
and hence, we obtain (1). 

Next, we prove (2). 
By \eqref{eq:L2 oncomp 1}, 
$\tilde *_\n \eta \in ( \Im \delta_{(\psi, \n)} \wFf )^\perp$ 
if and only if 
\[
\int_X F_\n \wedge d \xi \wedge \eta =0 
\quad \mbox{and} \quad 
\int_X db \wedge \left( \frac{1}{2} F_\n^2 +\psi \right) \wedge \eta =0
\]
for any $\xi \in \Om^3$ and $b \in \Om^1$. 
Since $F_\n$ and $\psi$ are closed, this is equivalent to 
\[
F_\n \wedge d \eta =0 
\quad \mbox{and} \quad 
\left( \frac{1}{2} F_\n^2 +\psi \right) \wedge d \eta =0, 
\]
and the proof is completed. 
\end{proof}

\begin{theorem}\label{thm:moduli MG2 generic}
Let $\n$ be a dDT connection with respect to 
the $G_2$-structure $\psi = * \varphi$ such that $d \psi =0.$ 
Suppose that one of the following conditions hold.
\begin{enumerate}
\item
The $G_2$-structure $\psi$ is torsion-free or nearly parallel. 
\item
The connection $\n$ satisfies $F_\n^3 \neq 0$ on a dense set of $X$. 
\end{enumerate}
Then, there exist open neighborhoods $\mathcal{U}_{1}\subset \Pp_\psi$ of $\psi$ and $\mathcal{U}_{2}\subset \Aa_{0}/\Gg_U$ of $[\nabla]$ 
such that, for every generic $\psi' \in \mathcal{U}_{1}$, 
each connected component $\mathcal{C}$ of the moduli space $\mathcal{M}_{G_2, \psi'}$ satisfying $\mathcal{C}\cap(\{\,\psi'\,\}\times\mathcal{U}_{2})\neq\emptyset$
is a smooth $b^1$-dimensional manifold which is homeomorphic to a torus. 
\end{theorem}

\begin{proof}
We first show that 
$\delta_{(\psi, \n)} \wFf: d \Om^3 \times \i \Om^1 \rightarrow \i d \Om^5$ is surjective if we assume (1) or (2). 
By Lemma \ref{lem:L2 oncomp}, 
we only have to show that 
$( \Im \delta_{(\psi, \n)} \wFf )^\perp = \{\, 0 \,\}$. 

Suppose that 
$\tilde *_\n \eta \in ( \Im \delta_{(\psi, \n)} \wFf )^\perp$ for $\eta \in d^{\tilde *_\n} \Om^2$. 
If (1) holds, we have $d \varphi = k \psi$ for $k \in \R$. 
By Lemma \ref{lem:L2 oncomp} (2), we have $\psi \wedge d \eta = 0$, 
and hence, $d \eta \in \Om^2_{14}$. 
Then, 
$$
0=k d \eta \wedge \psi = d \eta \wedge d \varphi
= d (d \eta \wedge \varphi)
= - d*d \eta, 
$$
which implies that $d^* d \eta=0$, and hence, $d \eta =0$. 
Since $\eta \in d^{\tilde *_\n} \Om^2$, it follows that $\eta = 0$. 

Suppose that (2) holds. Set $F = - \i F_\n \in \Om^2$ 
and $F=i(u) \varphi + F_{14}$, where $u$ is a vector field and $F_{14} \in \Om^2_{14}$. 
Since $\n$ is a dDT connection, we have $i(u) F_{14} =0$ by Lemma \ref{lem:1+F 00}. 
By Lemma \ref{lem:L2 oncomp} (2), 
we have $F \wedge d \eta =0$. 
Then, Lemma \ref{lem:kernel wedge} implies that $d \eta =0$ on a dense set of $X$. 
Since $d \eta$ is continuous, we see that $d \eta =0$ everywhere. 
Then, $\eta \in d^{\tilde *_\n} \Om^2$ implies that $\eta = 0$. 

Next, we show that $\widehat \Mm_{G_2}$ is smooth near $(\psi, [\n])$ by the implicit function theorem. 
As in Subsection \ref{sec:local str MG2}, the map 
\[
\Pp_\psi \times \Om^1_{d^{\tilde *_\n}} \rightarrow \Pp_\psi \times \Aa_{0}/\Gg_U, \quad
(\psi', a) \mapsto (\psi', [\nabla+ \i a \cdot \id_L])
\]
gives a homeomorphism from a neighborhood of $(\psi, 0) \in \Pp_\psi \times \Om^1_{d^{\tilde *_{\n}}}$ 
to that of $(\psi, [\nabla]) \in \Pp_\psi \times \Aa_0/\Gg_U$. 
Hence, a neighborhood of $(\psi, [\nabla])$ in $\widehat \Mm_{G_2}$ is homeomorphic to 
that of $(\psi, 0)$ in 
\[
\widehat \Ss_{(\psi, \n)} = \left\{\,(\psi',a) \in \Pp_\psi \times \Om^1_{d^{\tilde *_{\n}}}
\mid 
\wFf (\psi', \nabla + \i a \cdot \id_L) = 0\,
\right \}. 
\]
Thus, we only have to show that $\widehat \Ss_{(\psi, \n)}$ is smooth near $(\psi, 0)$. 
By the proof above, the map 
\begin{equation}\label{202106170935}
(\delta_{(\psi, \n)} \wFf) |_{d \Om^3 \times \Om^1_{d^{\tilde *_{\n}}}}: 
d \Om^3 \times \Om^1_{d^{\tilde *_{\n}}} \rightarrow \i d \Om^5
\end{equation}
is surjective if we assume (1) or (2). 
Then, since $\Im (\wFf) \subset \i d \Om^5$ by Lemma \ref{onimae G2 para}, we can apply the implicit function theorem 
(after the Banach completion) to 
\[
\Pp_\psi \times \Om^1_{d^{\tilde *_{\n}}} \rightarrow \i d \Om^5, \quad 
(\psi', a) \mapsto \wFf (\psi', \n + \i a \cdot \id_L)
\]
and we see that $\widehat \Ss_{(\psi, \n)}$ is smooth near $(\psi, 0)$. 
This means that there exist open neighborhoods $\mathcal{U}_{1}\subset \Pp_\psi$ of $\psi$ and $\mathcal{U}_{2}\subset \Aa_{0}/\Gg_U$ of $[\nabla]$ 
such that $\widehat \Mm_{G_2} \cap (\mathcal{U}_{1}\times \mathcal{U}_{2})$ admits a smooth manifold structure. 

Finally, we complete the proof. 
By the Sard--Smale theorem applied to the projection denoted by 
\[\Pi_{1}:\widehat \Mm_{G_2}\cap (\mathcal{U}_{1}\times \mathcal{U}_{2}) \rightarrow \mathcal{U}_{1}\subset\Pp_\psi,\]
for every generic $\psi' \in \mathcal{U}_{1}$, 
\begin{equation}\label{202106161628}
\delta_{(\psi',[\nabla'])}\Pi_{1}:T_{(\psi',[\nabla'])}\widehat\Mm_{G_2}\to T_{\psi'}\Pp_\psi=d \Om^3
\end{equation}
is surjective for all $[\nabla']\in \Pi_{1}^{-1}(\psi')$. 
This means that $\mathcal{M}_{G_2, \psi'}\cap (\{\,\psi'\,\}\times\mathcal{U}_{2})$ is a smooth manifold or an empty set.

For later use, we remark that the surjectivity of \eqref{202106161628} implies the surjectivity of the linearization of the projection before taking the equivalence class, 
that is, we can say that the projection map
\[
d \Om^3 \times \i \Om^1\supset \ker \delta_{(\psi', \n')} \wFf \to d\Om^3
\]
is also surjective since $\Om^1= d \Om^0 \oplus \Om^1_{d^{\tilde *_{\n}}}$ and 
$d \Om^0 \subset \ker \delta_{(\psi', \n')} \wFf$. 

Since \eqref{202106170935} is surjective for $(\psi,\n)$ and the surjectivity is an open condition, 
making $\mathcal{U}_{1}$ and $\mathcal{U}_{2}$ smaller if necessary, 
we can say that the linearization of $\wFf$ at $(\psi', \n')$ 
\[
\delta_{(\psi', \n')} \wFf : d \Om^3 \times \i \Om^1 \rightarrow \i d \Om^5 
\]
is also surjective. 
Then, by the argument such as in \cite[Lemma A.3.6]{McSa}, 
we see that 
\begin{align} \label{eq:diff psi'}
\delta_{\n'} \wFf (\psi',\,\cdot\,): \i \Om^1 \rightarrow \i d \Om^5 
\end{align}
is surjective. 

Now, we remark the regularity of elements in $\widehat \Mm_{G_2}$ as follows. 
By the proof of Theorem \ref{thm:moduli MG2}, 
the space $\Mm_{G_2, \psi} \subset \widehat \Mm_{G_2}$ 
can be seen as a space of solutions of an overdetermined elliptic equation around $[\nabla]$. 
Since to be overdetermined elliptic is an open condition, 
shrinking $\mathcal{U}_1$ and $\mathcal{U}_2$ if necessary, 
the space $\widehat \Mm_{G_2} \cap (\mathcal{U}_{1}\times \mathcal{U}_{2})$ can also be seen as a space of solutions 
of an overdetermined elliptic equation. 
In particular, elements of $\widehat \Mm_{G_2} \cap (\mathcal{U}_{1}\times \mathcal{U}_{2})$ are smooth.

Then, the surjectivity of \eqref{eq:diff psi'} (in the Banach setting) 
implies the surjectivity in the smooth setting. 
In other words, we have $H^2(\#_{\psi', \n'})=\{\,0\,\}$, 
where $(\#_{\psi', \n'})$ is the deformation complex 
with respect to the $G_2$-structure $\psi'$ and the dDT connection $\n'$ for $\psi'$ 
defined as in \eqref{cpxfordDT}. 
Then, by Theorem \ref{newtheorem}, 
we can say that the connected component of $\mathcal{M}_{G_2, \psi'}$ including $[\n']$ is homeomorphic to a $b^1$-dimensional torus.
\end{proof}
\subsection{The orientation of $\Mm_{G_2}$}
If every $[\nabla]\in \Mm_{G_{2}}$ satisfies $H^2(\#_\nabla) = \{\, 0 \,\}$, 
$\Mm_{G_{2}}$ is orientable because 
each connected component, which is a torus by Theorem \ref{newtheorem}, 
is orientable. 
However, this does not imply that we can pick a canonical orientation of $\Mm_{G_{2}}$ 
because we can take any orientation for each connected component. 
In this subsection, we show how to pick a canonical orientation of $\Mm_{G_{2}}$, 
which is considered to be important as explained in Section \ref{sec1}. 
Actually, we prove that such a canonical orientation is given by an orientation of the determinant of a Fredholm operator introduced below. 

Recall that $X^7$ is a compact 7-manifold with a coclosed $G_2$-structure $\varphi$. 
For each $\n \in \Ff_{G_2}^{-1}(0)$, set 
\begin{align} \label{cancpxfordDT G22}
D \eqref{cpxfordDT} 
= (d (\tilde *_\n \tilde \varphi_\n \wedge \bullet), d^*): \i \Om^1 \rightarrow \i d \Om^5 \oplus \i \Om^0, 
\end{align}
which is the two term complex associated with \eqref{cpxfordDT}. 
Note that $d^*$ is  the codifferential with respect to the metric 
induced from the initial $G_2$-structure $\varphi$. 
If $F_\n=0$, $D \eqref{cpxfordDT}$ becomes 
\begin{align} \label{cancpxfordDT G23}
D = (*\varphi \wedge d, d^*): \i \Om^1 \rightarrow \i d \Om^5 \oplus \i \Om^0. 
\end{align}

Set 
$$
\Aa'_0 = \{ \n \in \Aa_0 \mid \la F_\n^2, * \varphi \ra \neq -2 \} 
\quad \mbox{and} \quad 
\Bb'_0 = \Aa'_0/\Gg_U. 
$$
Note that $\Mm_{G_2} \subset \Bb'_0$. 
For any $\n \in \Aa'_0$, we can define a $G_2$-structure $\tilde \varphi_\n$ by \eqref{eq:newG2str}. 
Then, we define an operator 
$D \eqref{cpxfordDT}: \i \Om^1 \rightarrow \i d \Om^5 \oplus \i \Om^0$ 
as \eqref{cancpxfordDT G22}. 
We first show the following. 

\begin{lemma} \label{lem:FredG2}
For any $\n \in \Aa'_0$, $D \eqref{cpxfordDT}$ is a Fredholm operator. 
\end{lemma}
To be precise, 
we say that $D \eqref{cpxfordDT}$ is a Fredholm operator 
if it is Fredholm after the Banach completion. 
Since we are only interested in the kernel and the cokernel of $D \eqref{cpxfordDT}$, 
which consist of smooth forms,  
we use the above notation for simplicity. 
Lemma \ref{lem:FredG2} follows from the following more general statement.

\begin{lemma} \label{lem:FredG2 gen}
For any (not necessarily coclosed) $G_2$-structure $\phi \in \Om^3$ on $X$, 
set 
\[
D'' \eqref{cancpx G2} 
= (d(*_\phi \phi \wedge \bullet), d^*): \i \Om^1 \rightarrow \i d \Om^5 \oplus \i \Om^0, 
\]
where $*_\phi \phi$ is the Hodge dual of $\phi$ with respect to 
the induced metric and orientation from $\phi$ 
and $d^*$ is the codifferential with respect to the metric 
induced from the initial $G_2$-structure $\varphi$. 
Then, 
$D'' \eqref{cancpx G2}$ is Fredholm. 
\end{lemma}

Then, since $D \eqref{cpxfordDT} = D''(\flat_{\tilde \varphi_\n})$, 
Lemma \ref{lem:FredG2} follows. 

\begin{proof}
First, define the following sequence: 
\begin{equation}\label{cancpx G2}
\begin{aligned}
0 \rightarrow \i \Om^0 
\stackrel{d} \rightarrow \i \Om^1 
\stackrel{*_\phi \phi \wedge d} \longrightarrow \i \Om^6
\stackrel{d} \rightarrow \i \Om^7
\rightarrow 0. 
\end{aligned}
\tag{$\flat_{\phi}$} 
\end{equation}
By Lemma \ref{diamcomm2}, 
the sequence of the symbols is isomorphic to 
that of the canonical complex \eqref{eq:can cpx G2}. 
Then, the sequence of the symbols of \eqref{cancpx G2} is exact,  and hence, 
the associated two term complex 
\begin{align*}
D \eqref{cancpx G2}&: 
\i \Om^1 \oplus \i \Om^7 \rightarrow \i \Om^6 \oplus \i \Om^0, \\
D \eqref{cancpx G2}& (\alpha, \gamma) = (*_\phi \phi \wedge d \alpha + d^* \gamma, d^* \alpha) 
\end{align*}
is elliptic. 
Since 
\begin{align*} 
D' \eqref{cancpx G2}&: 
\i \Om^1 \oplus \i \Om^7 \rightarrow \i \Om^6 \oplus \i \Om^0, \\
D' \eqref{cancpx G2}& (\alpha, \gamma) = (d(*_\phi \phi \wedge \alpha) + d^* \gamma, d^* \alpha)
\end{align*}
has the same symbols as $D \eqref{cancpx G2}$, 
$D' \eqref{cancpx G2}$ is also elliptic.  
Then, the Hodge decomposition implies that  
$$
\ker D' \eqref{cancpx G2} \cong \ker D'' \eqref{cancpx G2} \oplus \R \vol_\varphi, \quad
{\rm Coker} D' \eqref{cancpx G2} \cong {\rm Coker} D'' \eqref{cancpx G2} \oplus \Hh^6, 
$$
where 
$\vol_\varphi$ and $\Hh^6$
are the volume form and the space of harmonic 6-forms 
induced from the initial $G_2$-structure $\varphi$, respectively. 
Hence, we see that $D'' \eqref{cancpx G2}$ is Fredholm. 
\end{proof}

By Lemma \ref{lem:FredG2}, we can 
define the determinant line bundle $\hat \Ll \rightarrow \Aa'_0$ by 
$$
\hat \Ll = \bigsqcup_{\n \in \Aa'_0} \det D \eqref{cpxfordDT}
= \bigsqcup_{\n \in \Aa'_0}
\Lambda^{{\rm top}} \ker D \eqref{cpxfordDT} 
\otimes (\Lambda^{{\rm top}} {\rm Coker} D \eqref{cpxfordDT} )^*.
$$
Since the curvature 2-form $F_\n$ is invariant under the action of $\Gg_U$, 
$\hat \Ll$ induces the line bundle $\Ll \rightarrow \Bb'_0$. 
We first show that $\Ll$ is a trivial line bundle.

\begin{proposition} \label{prop:trivlb G2}
We have an isomorphism 
$$
\Ll \cong \Bb'_0 \times \det D, 
$$ 
where $D$ is defined by \eqref{cancpxfordDT G23}.  
Hence, $\Ll$ is trivial. 
In particular, if we choose an orientation of $\det D$, 
we obtain a canonical orientation of $\Ll$. 
\end{proposition}

\begin{proof}
First note that 
$\{ D \eqref{cpxfordDT}: \i \Om^1 
\rightarrow \i d \Om^5 \oplus \i \Om^0 \}_{\n \in \Bb'_0}$ 
defines a $\Bb'_0$-family of 
Fredholm operators in the sense of \cite[Definition 2.4]{JU}. 
Note that $\Bb'_0$ is paracompact and Hausdorff
because it is metrizable by Remark \ref{rem:metric B0 G2}. 
For $[\n] \in \Bb'_0$ 
and $s \in [0,1]$, define a $G_2$-structure $\tilde \varphi_\n (s)$ by 
$$
\tilde \varphi_\n (s) = 
\left( s \left | 1 + \frac{1}{2} \la F_\n^2, * \varphi \ra \right |^{-3/4} +(1-s) \right) 
(\id_{TX} + s (- \i F_\n)^\sharp)^* \varphi, 
$$
which satisfies $\tilde \varphi_\n (0)=\varphi$ and $\tilde \varphi_\n (1) = \tilde \varphi_\n$. 
Then, by Lemma \ref{lem:FredG2 gen}, 
$D'' (\flat_{\tilde \varphi_\n (s)})$ is Fredholm for any $[\n] \in \Bb'_0$ and $s \in [0,1]$. 
Hence, the family 
$\{ D'' (\flat_{\tilde \varphi_\n (s)}): 
\i \Om^1 \rightarrow \i d \Om^5 \oplus \i \Om^0 \}_{\n \in \Bb'_0, s \in [0,1]}$ 
gives a homotopy between 
$\{ D \eqref{cpxfordDT}: \i \Om^1 
\rightarrow \i d \Om^5 \oplus \i \Om^0 \}_{\n \in \Bb'_0}$ 
and 
$\{ D: \i \Om^1 
\rightarrow \i d \Om^5 \oplus \i \Om^0 \}_{\n \in \Bb'_0}$.   
Then, by \cite[Lemma 2.6]{JU}, we obtain 
$$
\bigsqcup_{[\n] \in \Bb'_0} \det D \eqref{cpxfordDT}
\cong 
\bigsqcup_{[\n] \in \Bb'_0} \det D 
\cong 
\Bb'_0 \times \det D, 
$$
and the proof is completed. 
\end{proof}

Then, using this proposition, we show that $\Mm_{G_2}$ admits a canonical orientation. 

\begin{corollary} \label{cor:oriG2}
Suppose that $H^2(\#_\nabla) = \{\, 0 \,\}$ for any $[\n] \in \Mm_{G_2}$. 
Then, $\Mm_{G_2}$ is a $b^{1}$-dimensional manifold which is homeomorphic to the disjoint union of tori 
and it admits a canonical orientation determined by an orientation of $\det D$, 
where $D$ is defined by \eqref{cancpxfordDT G23}. 
\end{corollary}

\begin{proof}
The former half of the statement is clear by Theorem \ref{newtheorem}. 
The proof for the later half on the canonical orientation is as follows. 
Since $H^2(\#_\nabla) = \{\, 0 \,\}$, which is equivalent to ${\rm Coker} D (\#_\nabla) = \R$, 
for any $[\n] \in \Mm_{G_2}$, the tangent space $T_{[\n]} \Mm_{G_2}$ at $[\n] \in \Mm_{G_2}$ is 
identified with $\ker D \eqref{cpxfordDT}$. 
Thus, denoting by $\iota: \Mm_{G_2} \hookrightarrow \Bb'_0$ the inclusion, 
we see that 
$$
\Lambda^{{\rm top}} T \Mm_{G_2} 
\cong \iota^* \Ll 
\cong \Mm_{G_2} \times 
\det D
$$
by Proposition \ref{prop:trivlb G2} and the proof is completed. 
\end{proof}
\section{The moduli space of dHYM connections}\label{sec:defofdHYM}
In this section, we switch the subject of the study to deformed Hermitian Yang--Mills connections from deformed Donaldson--Thomas connections. 
More strongly than the case of dDT connections, 
we can determine the structure of the moduli space $\Mm$ of deformed Hermitian Yang--Mills connections on a 
K\"ahler manifold, that is, we show that $\Mm$ \emph{itself} is a $b^{1}$-dimensional torus. 

\subsection{Preliminaries to dHYM connections} \label{sec:prel dHYM}
First, we recall the definition of dHYM connections. 
Let $(X,\omega,g,J)$ be a compact 
K\"ahler manifold with $\mathop{\mathrm{dim}_{\mathbb{C}}}X=n$. We always denote by $\omega$, $g$ and $J$ 
the K\"ahler form, the Riemannian metric and the complex structure on $X$, respectively. 
Let $L \to X$ be a smooth complex line bundle with a Hermitian metric $h$. 

\begin{definition}
For a constant $\theta\in\mathbb{R}$, 
a \emph{deformed Hermitian Yang--Mills (dHYM) connection} with phase $e^{\sqrt{-1}\theta}$ is 
a Hermitian connection $\nabla$ of $(L,h)$ satisfying 
\[
F^{0,2}_{\nabla}=0
\quad\mbox{and}\quad
\mathop{\mathrm{Im}}\left(e^{-\i \theta} (\omega + F_\nabla)^n\right)=0, 
\]
where 
$F_\n$ is the curvature 2-form of $\n$, which we consider as a $\i \R$-valued closed 2-form on $X$, 
and $F^{0,2}_\nabla$ is its $(0,2)$-part.
\end{definition}

For the rest of this subsection, we assume that 
$\n$ is a Hermitian connection of $(L,h)$ satisfying $F^{0,2}_\nabla=0$. 
We remark that this implies that $F_{\nabla}$ is a $\sqrt{-1}\mathbb{R}$-valued $(1,1)$-form. 
By using $\nabla$, we define a Hermitian metric $\eta_{\nabla}$ on $X$ 
following \cite{JY} and show some preliminary formulas. 

\begin{definition}
Define a Riemannian metric $\eta_{\nabla}$ on $X$ 
by
\[
\eta_{\nabla}(u,v)=g(u,v)+ g\left(\left(-\sqrt{-1}i(u)F_{\nabla}\right)^{\sharp},\left(-\sqrt{-1}i(v)F_{\nabla}\right)^{\sharp}\right)
\]
for $u,v\in TX$, where $A^{\sharp}$ (for any $A\in \Omega^{1}$) is defined by $g(A^{\sharp},w)=A(w)$ for any $w\in TX$, 
and define its associated 2-form $\omega_{\nabla}$ by 
\[\omega_{\nabla}(u,v):=\eta_{\nabla}(Ju,v). \]
\end{definition}
Here, we follow \cite{JY} for the notation of the above metric. 
We put the proof of that $\eta_{\nabla}$ is a Hermitian metric. 
Fix a point $p\in X$ arbitrarily. 
Since $F_\n \in \Om^{1,1}$, 
we see that $F_{\nabla}(Ju,Jv)=F_{\nabla}(u,v)$ for $u,v\in T_{p}X$, 
and hence, $-\sqrt{-1}F_\n^\sharp$ is complex linear. 
By definition, $-\sqrt{-1}F_\n^\sharp$ is skew-symmetric. 
Then, it follows that there exists an orthonormal basis 
$u_{1},\dots,u_{n}, v_{1}\dots,v_{n}$ with relation $v_{i}=Ju_{i}$ and real constants 
$\lambda_{1},\dots,\lambda_{n}$ such that 
\[-\sqrt{-1}F_\n^\sharp
=\sum_{i=1}^{n} \lambda_{i} (u^i \otimes v_i-v^i \otimes u_i )\]
at $p$, where $\{\, u^{i},v^{i} \,\}_{i=1}^{n}$ is the dual basis of $\{\, u_{i},v_{i} \,\}_{i=1}^{n}$. 
Then, in terms of this basis, 
we have 
$g =\sum_{i=1}^{n} (u^i \otimes u^i + v^i \otimes v^i), \omega=\sum_{i=1}^{n} u^{i}\wedge v^{i},$
\begin{equation} \label{omandometa}
\eta_{\nabla}=\sum_{i=1}^{n}(1+\lambda_{i}^2)(u^{i}\otimes u^{i}+v^{i}\otimes v^{i}) 
\quad\mbox{and}\quad
\omega_{\nabla}=\sum_{i=1}^{n}(1+\lambda_{i}^2)(u^{i}\wedge v^{i}) 
\end{equation}
at $p$. Then, this expression implies that $\eta_\n$ is Hermitian.

\begin{definition}
Define $\zeta_{\nabla}:X\to \mathbb{C}$ by 
\[
(\omega+F_{\nabla})^{n}=\zeta_{\nabla}\omega^{n}. 
\]
As we see $|\zeta_{\nabla}| \geq 1$ later, there exist $r_{\nabla}:X\to \mathbb{R}^{+}$ 
and $\theta_{\nabla}:X\to\mathbb{R}/2\pi\mathbb{Z}$ such that $\zeta_{\nabla}=r_{\nabla}e^{\sqrt{-1}\theta_{\nabla}}$. We call $r_{\nabla}$ and $\theta_{\nabla}$ the \emph{radius function} and the \emph{angle function} of $\nabla$, respectively. 
\end{definition}

\begin{lemma}
We have $r_\n \geq 1$, 
\begin{align}
&n! \mathop{\mathrm{vol}_{\eta_\nabla}} = \om_\nabla^n= r_\nabla^2 \om^n, \label{eq:vol nab2} \\
&\mathop{\mathrm{Im}} \left(\sqrt{-1} e^{- \sqrt{-1} \theta_{\nabla}} (\omega + F_\nabla)^{n-1}\right) = \frac{1}{r_{\nabla}} \om_\nabla^{n-1}, \label{imissome}
\end{align}
where $\mathop{\mathrm{vol}_{\eta_\nabla}}$ is the volume form of $\eta_\n$. \end{lemma}

\begin{proof}
By \eqref{omandometa}, we have 
\begin{equation}\label{nvoleta}
n! \mathop{\mathrm{vol}_{\eta_\nabla}}=\omega_{\nabla}^{n}=
n!\prod_{i=1}^{n}(1+\lambda_{i}^2)\bigwedge_{j=1}^{n}( u^{j}\wedge v^{j})
=\prod_{i=1}^{n}(1+\lambda_{i}^2)\omega^{n}. 
\end{equation}
Next, we write $\zeta_{\nabla}$ in terms of $\lambda_{i}$. 
Since 
$F_{\nabla}=\sum_{i=1}^{n}(\sqrt{-1}\lambda_{i})u^{i}\wedge v^{i}$, 
we have 
\begin{equation}\label{omegaplusf}
\omega+F_{\nabla}=\sum_{i=1}^{n}(1+\sqrt{-1}\lambda_{i})u^{i}\wedge v^{i}
=\sum_{i=1}^{n}r_{i}e^{\sqrt{-1}\theta_{i}}u^{i}\wedge v^{i}, 
\end{equation}
where $r_{i}:=\sqrt{1+\lambda_{i}^2}$ and $\theta_{i}:=\arctan \lambda_{i}$. 
Then, we have 
\[
(\omega+F_{\nabla})^{n}=\left(\prod_{i=1}^{n}r_{i}\right)e^{\sqrt{-1}\sum_{k=1}^{n}\theta_{k}}\omega^{n}. 
\]
Since $\zeta_{\nabla}$ is the coefficient of $\omega^{n}$, 
we see that the radius function and the angle function are written as
\begin{align}
r_{\nabla}=&|\zeta_{\nabla}|=\prod_{i=1}^{n}r_{i}=\prod_{i=1}^{n}\sqrt{1+\lambda_{i}^2},  \label{radifunc} \\
\theta_{\nabla}=&\arg\zeta_{\nabla}=\sum_{k=1}^{n}\theta_{k}=\sum_{k=1}^{n}\arctan\lambda_{k} \mod 2\pi. \label{angfunc}
\end{align}
Then, by \eqref{radifunc}, it is clear that $r_{\nabla} \geq1$. 
By \eqref{nvoleta} and \eqref{radifunc}, we also obtain \eqref{eq:vol nab2}. 

Next, we prove \eqref{imissome}. 
By \eqref{omegaplusf}, we have 
\[(\omega+F_{\nabla})^{n-1}=(n-1)!\sum_{i=1}^{n}\hat{r}_{i}e^{\sqrt{-1}\hat{\theta}_{i}}\hat{\Omega}_{i}, \]
where
\[\hat{r}_{i}:=\prod_{k\neq i}^{n}\frac{1}{r_{k}},\quad\hat{\theta}_{i}:=\sum_{k\neq i}\theta_{k},\quad \hat{\Omega}_{i}:=\prod_{k\neq i}r_{k}^2(u^{k}\wedge v^{k}). \]
By \eqref{angfunc}, we have 
\[\mathop{\mathrm{Im}}\left(\sqrt{-1}e^{-\sqrt{-1}\theta_{\nabla}}e^{\sqrt{-1}\hat{\theta}_{i}}\right)
=\mathop{\mathrm{Im}}\left(\sqrt{-1}e^{-\sqrt{-1}\theta_{i}}\right)=\frac{1}{r_{i}}. \]
Thus, with the second identity of \eqref{omandometa}, we obtain 
\[
\begin{aligned}
&\mathop{\mathrm{Im}} \left(\sqrt{-1} e^{- \sqrt{-1} \theta_{\nabla}} (\omega + F_\nabla)^{n-1}\right) \\
=&(n-1)!\sum_{i=1}^{n}\hat{r}_{i}\mathop{\mathrm{Im}}\left(\sqrt{-1}e^{-\sqrt{-1}\theta_{\nabla}}e^{\sqrt{-1}\hat{\theta}_{i}}\right)\hat{\Omega}_{i}
=\frac{1}{r_{\nabla}}\omega_{\nabla}^{n-1}, 
\end{aligned}
\]
where we use $\om_\n^{n-1} = (n-1)! \sum_{i=1}^n \hat{\Omega}_i$. 
This is \eqref{imissome}. 
\end{proof}

Define a new Hermitian metric on $X$ (which is conformal to  $\eta_\n$) by  
\[
\tilde \eta_\n = \left(\frac{1}{r_{\nabla}} \right)^{1/(n-1)}\eta_\n.
\]
Let $\tilde \om_\n = \tilde \eta_\n (J\,\cdot\, , \, \cdot\,)$ be the associated 2-form. 
Denote by $\tilde{\ast}_\n$, $d^{\tilde{\ast}_\n}$ and $d^{\tilde{\ast}_\n}_{c}$ 
the Hodge star operator, the formal adjoint of $d$ and $d_c$ with respect to $\tilde{\eta}_{\n}$, respectively.

\begin{corollary}\label{balanced}
If $\nabla$ is a dHYM connection with phase $e^{\sqrt{-1}\theta}$, 
then Hermitian metric $\tilde{\eta}_\nabla$ is balanced, 
that is, $d \tilde \om_\nabla^{n-1} = 0$. 
\end{corollary}

\begin{proof}
In this case, $\theta_{\nabla}$ is the constant $\theta$ in \eqref{imissome}. 
Then, taking an exterior derivative of the both hand side of \eqref{imissome} with 
$(1/r_{\nabla})\omega_{\nabla}^{n-1}=\tilde{\omega}_{\nabla}^{n-1}$ implies $d\tilde{\omega}_{\nabla}^{n-1}=0$. 
\end{proof}

\begin{proposition}\label{1stderiofF}
If $\nabla$ is a dHYM connection with phase $e^{\sqrt{-1}\theta}$, we have 
\[\frac{\tilde{\omega}_{\nabla}^{n-1}}{(n-1)!}\wedge db=-\tilde{\ast}_{\nabla} d^{\tilde{\ast}_{\nabla}}_c b, \]
for any $b\in \Omega^{1}$. 
\end{proposition}

\begin{proof}
Since $d\tilde{\omega}_{\nabla}^{n-1}=0$ by Corollary \ref{balanced}, we have 
\[\tilde{\omega}_{\nabla}^{n-1}\wedge db=d(\tilde{\omega}_{\nabla}^{n-1}\wedge b)=(n-1)!d\tilde{\ast}_{\nabla}Jb=-(n-1)!\tilde{\ast}_{\nabla} d^{\tilde{\ast}_{\nabla}}_c b, \]
where the second equality follows from Lemma \ref{lem:gen id} 
and the last equality follows from \eqref{A1.5?}. 
\end{proof}
\subsection{The moduli space of dHYM connections}
In this subsection, we study the moduli space of dHYM connections. 
To give the definition of the moduli space of dHYM connections, 
let $\Gg_U$ 
be the group of unitary gauge transformations of $(L,h)$. 
As in the case of Subsection \ref{sec:def dDT}, 
$\Gg_{U}$ acts on Hermitian connections in a standard way and 
the $\Gg_U$-orbit through a Hermitian connection $\n$ is given by $\n + \Kk_U \cdot \id_L$, 
where $\Kk_U$ is defined by \eqref{eq:Gu orbit G2}. 

Since the curvature 2-form $F_\nabla$ is invariant under the action of $\Gg_U$, 
the $\Gg_U$-action preserves the set of dHYM connections with a fixed phase and 
we get the following definition. 

\begin{definition}
The {\em moduli space} $\Mm$ of deformed Hermitian Yang--Mills connections  
with phase $e^{\sqrt{-1}\theta}$ of $(L,h)$ is given by 
\[
\Mm= \{\,\mbox{dHYM connections with phase }e^{\sqrt{-1}\theta}\,\}/\Gg_U. 
\]
\end{definition}
Assume that $X$ is compact. 
Then, the following is a main theorem in this section. 

\begin{theorem} \label{thm:smooth M}
If $\Mm \neq \emptyset$, then the moduli space $\Mm$ is homeomorphic to a $b^{1}$-dimensional torus, where $b^1$ is the first Betti number of $X$. 
\end{theorem}

The key of the proof is the observation that $\nabla+\sqrt{-1}\alpha\cdot\mathrm{id}_{L}$, where $\alpha$ is a harmonic 1-form, is also a dHYM connection for a fixed dHYM connection $\nabla$. This roughly implies that the moduli space includes a $b^{1}$-dimensional torus, 
and actually we can say that the moduli space coincides with the torus.

\begin{proof}
Fix $[\nabla]\in\mathcal{M}$. Define a map 
\[\mathcal{G}:H^{1}(X,\mathbb{R})/2\pi H^{1}(X,\mathbb{Z})\to\mathcal{M}\]
by
\[\widetilde{[\alpha]}\mapsto [\nabla+\sqrt{-1}\alpha\cdot\mathrm{id}_{L}], \]
where $[\alpha]$ is the cohomology class of a closed 1-form $\alpha$, 
$\widetilde{[\alpha]}$ means the equivalence class of 
$[\alpha]\in H^{1}(X,\mathbb{R})$ in $H^{1}(X,\mathbb{R})/2\pi H^{1}(X,\mathbb{Z})$ 
and $[\nabla+\sqrt{-1}\alpha\cdot\mathrm{id}_{L}]$ is the gauge equivalence class of $\nabla+\sqrt{-1}\alpha\cdot\mathrm{id}_{L}$. 

First, we prove the well-definedness of $\mathcal{G}$. 
Assume that $[\alpha_{1}]-[\alpha_{2}]\in 2\pi H^{1}(X,\mathbb{Z})$. 
Then, by \cite[Lemma 4.1]{KY2}, there exists $f_{1}\in C^{\infty}(X,S^{1})$ such that $[\alpha_{1}]-[\alpha_{2}]=-\sqrt{-1}[f_{1}^{-1}df_{1}]$. 
Thus, $\alpha_{1}-\alpha_{2}+\sqrt{-1}f_{1}^{-1}df_{1}=df_{2}$ for some $f_{2}\in \Om^{0}$. 
Putting $f_{3}:=f_{1}\cdot e^{\sqrt{-1}f_{2}}\in C^{\infty}(X,S^{1})$, we have $\sqrt{-1}\alpha_{1}-\sqrt{-1}\alpha_{2}=f_{3}^{-1}df_{3}$. 
This means that $\nabla+\sqrt{-1}\alpha_{1}\cdot\mathrm{id}_{L}$ and $\nabla+\sqrt{-1}\alpha_{2}\cdot\mathrm{id}_{L}$ are 
equivalent by the $U(1)$-gauge action. Thus, the well-definedness of $\mathcal{G}$ is checked. 

We will end the proof by checking that $\mathcal{G}$ is bijective. 
To prove the injectivity of $\mathcal{G}$, assume that $[\nabla+\sqrt{-1}\alpha_{1}\cdot\mathrm{id}_{L}]=[\nabla+\sqrt{-1}\alpha_{2}\cdot\mathrm{id}_{L}]$. 
Since this means that these two representing connections are $U(1)$-gauge equivalent, we see that 
$\sqrt{-1}\alpha_{1}-\sqrt{-1}\alpha_{2}=f^{-1}df$ for some $f\in\ C^{\infty}(X,S^{1})$ (see, for instance, \eqref{eq:Gu orbit G2}). 
Taking the cohomology class, we have $[\alpha_{1}]-[\alpha_{2}]=-\sqrt{-1}[f^{-1}df]$ in $H^{1}(X,\mathbb{R})$. 
Then, by using \cite[Lemma 4.1]{KY2} again, we see that $-\sqrt{-1}[f^{-1}df]$ is actually an element in $2 \pi H^{1}(X,\mathbb{Z})$. 
Thus, we have $\widetilde{[\alpha_{1}]}=\widetilde{[\alpha_{2}]}$ and the injectivity is proved. 

We give the proof of surjectivity. 
Fix $[\nabla']\in \mathcal{M}$ arbitrarily. 
Then, $\nabla'$ can be written as $\nabla'=\nabla+\sqrt{-1}\alpha\cdot \mathrm{id}_{L}$ 
with some 1-form $\alpha$ on $X$. 
Then, it suffices to prove that $\alpha$ is closed for the surjectivity of $\mathcal{G}$. 
By $F_{\nabla'}^{0,2}=0$ and $F_{\nabla}^{0,2}=0$, we have $(d\alpha)^{0,2}=0$. 
Then, by Lemma \ref{lem:ker F02}, we can decompose $\alpha$ as $\alpha=\beta+d_{c}f$ with some closed 1-form $\beta$ and $f\in \Om^{0}$. 
For a parameter $t\in[0,1]$,  set $\nabla_{t}:=\nabla+\sqrt{-1}(\beta+td_{c}f)\cdot \mathrm{id}_{L}$ and 
\[\Theta(t):=\arg\left(\left(\omega+F_{\nabla_{t}}\right)^{n}/\omega^{n}\right), \]
where $\Theta(t)$ is the angle function of $\nabla_{t}$. 
Since $F_{\nabla_{t}}=F_{\nabla}+t\sqrt{-1}dd_{c}f$, $\Theta(t)$ is actually written as 
\[\Theta(t):=\arg\left(\left(\omega+F_{\nabla}+t\sqrt{-1}dd_{c}f)\right)^{n}/\omega^{n}\right). \]
Since $F_{\nabla}$ and $F_{\nabla'}$ are dHYM connections with phase $e^{\sqrt{-1}\theta}$, we know that 
\[
\Theta(1)=\arg\left(\left(\omega+F_{\nabla'}\right)^{n}/\omega^{n}\right)
=\theta
\quad\mbox{and}\quad
\Theta(0)=\arg\left(\left(\omega+F_{\nabla}\right)^{n}/\omega^{n}\right)=\theta. 
\]
Thus, we have 
\begin{equation}\label{unifPDF}
0=\Theta(1)-\Theta(0)=\int_0^{1}\frac{d}{dt}\bigg|_{t=s}\Theta(t)ds=:\mathcal{D}(f). 
\end{equation}
To compute $\Theta(t)'$, taking the $t$-derivative of the identity $\left(\omega+F_{\nabla_{t}}\right)^{n}=r_{\nabla_{t}}e^{\sqrt{-1}\Theta(t)}\omega^{n}$, 
we have $n\sqrt{-1}\left(\omega+F_{\nabla_{t}}\right)^{n-1}\wedge dd_{c}f=(r_{\nabla_{t}})'e^{\sqrt{-1}\Theta(t)}\omega^{n}+\sqrt{-1}r_{\nabla_{t}}\Theta(t)'e^{\sqrt{-1}\Theta(t)}\omega^{n}$. 
Thus, multiplying the both hand sides with $e^{-\sqrt{-1}\Theta(t)}$ and taking its imaginary part 
with applying \eqref{imissome} and \eqref{eq:vol nab2}, one can see that 
\begin{equation}\label{Dfddc}
\mathcal{D}(f)=\int_{0}^{1}\frac{d}{dt}\bigg|_{t=s}\Theta(t)ds=\int_{0}^{1}\left(\frac{n\om_{\n_s}^{n-1} \wedge d d_c f}{\om_{\n_s}^{n}}\right)ds. 
\end{equation}
Fix $x\in X$, $\xi\in T^{*}_{x}X$ and $s\in [0,1]$. 
Express $\xi$ as $\xi=\sum_{i=1}^{n}(a_{i}u^{i}+b_{i}v^{i})$ for an orthonormal basis $u_{1},\dots, u_{n}, v_{1},\dots,v_{n}$ of $T_{x}X$ with respect to $g$ 
satisfying relation $v_{i}=Ju_{i}$ and $-\sqrt{-1}F_{\n_{s}}^\sharp
=\sum_{i=1}^{n} \lambda_{i} (u^i \otimes v_i-v^i \otimes u_i )$ for some $\lambda_{i}=\lambda_{i}(x,s)$. 
We remark that, by \eqref{profdc}, the principal symbol of $dd_{c}$ is given by 
\[\sigma_{\xi}(dd_{c})=\xi\wedge J^{-1}\xi=\sum_{i=1}^{n}(a_{i}^2+b_{i}^2)u^{i}\wedge v^{i}+C,\]
where the term $C$ is some linear combination of $u^{i}\wedge v^{j}$ ($i\neq j$), $u^{i}\wedge u^{j}$ and $v^{i}\wedge v^{j}$. 
Then, by \eqref{omandometa} and \eqref{nvoleta}, one can see that 
the principal symbol of $\mathcal{D}$ is given by 
\begin{equation}\label{symbofTheta}
\sigma_{\xi}(\mathcal{D})=\int_{0}^{1}\left(\sum_{i=1}^{n}\frac{a_{i}^2+b_{i}^2}{1+\lambda_{i}^2}\right)ds. 
\end{equation}
Since each $\lambda_{i}$ is a continuous function for $x\in X$ and $s\in[0,1]$ and now $X$ is compact, 
there exists a constant $\lambda$ such that $|\lambda_{i}(x,s)|\leq \lambda$ for all $1\leq i\leq n$, $x\in X$ and $s\in [0,1]$. 
Then, by \eqref{symbofTheta}, we have 
\[\sigma_{\xi}(\mathcal{D})\geq \frac{1}{1+\lambda^2}|\xi|^2_{g}\]
and this implies that the second order linear PDE 
\eqref{unifPDF} is uniformly elliptic. 
Then, the strong maximum principle implies that $f$ is constant. 
(This is a detailed version of an argument in the proof of \cite[Lemma 2.3]{CXY}, which is another proof of \cite[Theorem 1.1]{JY}. )
Thus, we have $\alpha=\beta+d_{c}f=\beta$ and this is a closed 1-form. 
The proof is completed. 
\end{proof}

\begin{remark}
The proof of Theorem \ref{thm:smooth M} does not extend to the dDT setting. 
In the proof of the surjectivity, 
we first use the condition that $F_{\n'}^{0,2}=0$, which implies that 
$\n'=\n+\i(\beta + d_c f) \cdot \id_L$, where $\beta$ is a closed 1-form and $f \in \Om^0$. 
Since the curvature is invariant under the addition of closed 1-forms, 
the essential part to consider is the direction of $d_c \Om^0$. 
Then, we can focus on the argument on functions and use powerful tools such as the maximum principle. 

In other words, we can say that 
the surjectivity is a consequence of \cite[Theorem 1.1]{JY} 
stating the uniqueness of dHYM metrics up to multiplication of a positive constant. 
(Here, a dHYM metric is a Hermitian metric of a holomorphic line bundle 
whose Chern connection satisfies the dHYM condition.) 
It is because by \cite[Remark 6.12]{KY2}, we can identify the direction of $d_c \Om^0$ 
with the space of Hermitian metrics on a holomorphic line bundle, 
and the condition for a metric to be a dHYM metric 
corresponds to the condition for the Chern connection to be a dHYM connection.

On the other hand, there is no corresponding condition to $F_{\n'}^{0,2}=0$ 
in the dDT setting, so we cannot take a nice space like $d_c \Om^0$. 
In particular, we cannot use tools such as the maximum principle, 
and hence, 
the proof of Theorem \ref{thm:smooth M} does not extend to the dDT setting. 
\end{remark}

Let $\Aa_0$ be the space of Hermitian connections of $(L, h)$. 
As in the case of dDT connections, we can consider the following deformation map 
$\Ff = (\Ff_1, \Ff_2): \Aa_0 \rightarrow \Om^{0, 2}\oplus \Om^{2n}$ by 
\begin{equation*}
\begin{aligned}
\Ff (\n) = (\Ff_1 (\n), \Ff_2 (\n))
= \left( -\sqrt{-1}F_\n^{0,2}, \ \mathop{\mathrm{Im}}\left(e^{-\i \theta} (\omega + F_\nabla)^n\right) \right). 
\end{aligned}
\end{equation*}
Then, we see that $\Mm=\Ff^{-1}(0)/\Gg_U$. 
It is trivial to consider the deformation theory since the structure of $\Mm$ is determined 
by Theorem \ref{thm:smooth M}. 
However, we consider the linearization of the deformation map $\Ff$ since 
we have an application in the next subsection. 

\begin{lemma} \label{lem:isom}
Let $\delta_\n \Ff$ be the linearization of $\Ff$ at $\n \in \Ff^{-1}(0)$. 
Denote by $d^{\tilde *_{\n}}$ the formal adjoint of $d$ with respect to $\tilde \eta_\n$. 
Then, the map 
\[
D:=(\delta_\n \Ff) |_{ d^{\tilde *_{\n}} \Om^2}: 
d^{\tilde *_{\n}} \Om^2 \rightarrow \pi^{0,2} (d \Om^1) \oplus d \Om^{2n-1}, 
\]
where $\pi^{0,2}(\,\cdot\,):=(\,\cdot\,)^{0,2}: \Om^2 \to \Om^{0,2}$ is the projection, 
is an isomorphism. 
\end{lemma}

\begin{proof}
We first show that the image of $\Ff$ is contained in $\pi^{0,2} (d \Om^1) \oplus d \Om^{2n-1}$. 
For any $\nabla' \in \Aa_0$, we know that $F_{\nabla'}-F_\nabla \in \i d \Om^1$. 
Then, it follows that 
\[
\Ff_1 (\nabla') = \Ff_1 (\nabla') - \Ff_1 (\nabla) \in \pi^{0,2} (d \Om^1), \quad
\Ff_2 (\nabla') = \Ff_2 (\nabla') - \Ff_2 (\nabla) \in d \Om^{2n-1}. 
\]
Thus, we see that the image of $D$ is contained in $\pi^{0,2} (d \Om^1) \oplus d \Om^{2n-1}$. 

By a similar computation as in \eqref{Dfddc} with Proposition \ref{1stderiofF}, we can see that 
\[
D(\alpha)=\left((d\alpha)^{0,2}, \ -n!\tilde *_{\n}d_{c}^{\tilde *_{\n}}\alpha \right) 
\]
for $\alpha \in d^{\tilde *_{\n}} \Om^2$. 
To prove the surjectivity of $D$, 
we omit the symbol $\sim$ over $\ast$ and the subscript $\nabla$ of $\ast$ for the simplicity of notations. 
First, note that for any $\alpha \in \Om^1$ 
\begin{equation} \label{eq:balanced 02}
\pi^{0,2} (d \alpha) = - \i \pi^{0,2} (d_c \alpha). 
\end{equation}
By \eqref{eq:Hodge d} and \eqref{eq:Hodge dc2}, we can decompose $\Om^{1}$ in two ways as 
\[
\Om^1= \Hh^1_{d_c} \oplus d_c \Om^0\oplus d_c^* \Om^{2}\quad\mbox{and}\quad
\Om^1= \Hh^1_{d}\oplus d \Om^0 \oplus d^* \Om^{2}, 
\]
where $\Hh^1_{d_c}$ and $\Hh^1_{d}$ are 
the space of $d_c$ harmonic 1-forms and $d$ harmonic 1-forms, respectively. 
Denote by 
\[
\pi_{\Hh}: \Om^1 \rightarrow \Hh^1_{d}, \quad
\pi_{d}: \Om^1 \rightarrow d \Om^0, \quad
\pi_{d^*}: \Om^1 \rightarrow d^* \Om^{2}
\]
the projections about the second decomposition. 
By \eqref{eq:balanced 02}, it follows that 
\[
\pi^{0, 2} (d \Om^1) 
= \i \pi^{0, 2} (d_c \Om^1) = \i \pi^{0, 2} (d_c d_c^* \Om^2). 
\]
By \eqref{eq:balanced 02} again, we have 
\[
\i \pi^{0, 2} (d_c d_c^* \Om^2) 
= \pi^{0, 2} \left( d d_c^* \Om^2 \right) 
= \pi^{0, 2} \left( d (\pi_{d^*} (d_c^* \Om^2) ) \right).
\]
Hence, we obtain 
\[
\pi^{0, 2} (d \Om^1) 
= \pi^{0, 2} \left( d (\pi_{d^*} (d_c^* \Om^2) ) \right). 
\]
On the other hand, we have 
\[
d \Om^{2n-1}
= * d_c^* \Om^1 
= * d_c^* d_c \Om^0 
= * d_c^* \left( \pi_{d^*} (d_c \Om^0) \right), 
\]
where the last equality follows from 
the fact $Z^{1} \subset \ker d_c^*$, which immediately follows from Proposition \ref{1stderiofF}. 
Thus, any element of $\pi^{0, 2} (d \Om^1) \oplus d \Om^{2n-1}$ 
is of the form 
$\left(\pi^{0, 2} \left( d \beta \right), * d_c^* \gamma \right)$
for 
$\beta \in d_c^* \Om^2$ and $\gamma \in d_c \Om^0$. 
Then, setting $\alpha = \pi_{d^*} (\beta) + \pi_{d^*} (\gamma)/(-n!) \in d^* \Om^2$, 
we have 
\[
\pi^{0,2}(d \alpha) 
= \pi^{0,2} (d \beta + d \gamma/(-n!)) 
= \pi^{0,2} (d \beta) - \i  \pi^{0,2} (d_c \gamma/(-n!)) 
= \pi^{0,2} (d \beta)
\]
and
\[
-n! * d_c^* \alpha 
= * d_c^* (-n! \beta + \gamma) = * d_c^* \gamma, 
\]
which implies that $D$ is surjective. 

To prove the injectivity of $D$, suppose $D(\alpha)=0$ for $\alpha\in d^{\ast}\Omega^{2}$. 
By the proof of Theorem \ref{thm:smooth M}, $\Ff^{-1}(0)= \n + Z^1$, where $Z^1$ is the space of closed 1-forms. 
Hence, we see that 
$\alpha \in \ker \delta_\n \Ff = T_\n \Ff^{-1}(0) = Z^1$. By the Hodge decomposition, we have $Z^1 \cap d^{\ast}\Omega^{2} = \{\, 0 \,\}$, 
which implies that $\alpha=0$. 
\end{proof}
\subsection{Obstructions to the existence of dHYM connections}
Let $X$ be a compact manifold with $\mathrm{dim}_{\mathbb{R}}X=2n$ and $L\to X$ be a smooth complex line bundle. 
Let $B\subset \mathbb{R}^{m}$ be an open ball with $0\in B$ 
and assume that a smooth family of K\"ahler structures $\{\,(\omega_{t},g_{t},J_{t})\mid t\in B\,\}$ on $X$, 
Hermitian metrics $\{\,h_{t}\mid t\in B\,\}$ of $L$ and constants $\{\,\theta_{t}\in\mathbb{R}\mid t\in B\,\}$ are given. 
Note that the Hermitian metrics on $L$ can vary unlike Subsection \ref{202106181615}. 
Further assume that a dHYM connection $\nabla_{0}$ 
of $(L,h_{0})$ on $(X,\omega_{0},g_{0},J_{0})$ with phase $e^{\sqrt{-1}\theta_{0}}$ is also given. 
In this subsection, we study the following question:
Can we extend $\nabla_{0}$ to a smooth family of dHYM connections 
$\n_t$ of $(L,h_{t})$ on $(X,\omega_{t},g_{t},J_{t})$ with phase $e^{\sqrt{-1}\theta_{t}}$
for $t \in B$?
The answer to the similar question for moduli spaces of special Lagrangian submanifolds is affirmative (see \cite[Theorem 3.21]{Marshall}) and so in this case. 

First, we pay attention to the necessary condition. 

\begin{proposition}
If there exists a smooth family of connections $\{\,\nabla_{t}\mid t\in B\,\}$ of $L$ with $\nabla_{t}|_{t=0}=\nabla_{0}$ so that $\nabla_{t}$ is a dHYM connection of $(L,h_{t})$ with phase $e^{\sqrt{-1}\theta_{t}}$ with respect to $(\omega_{t},g_{t},J_{t})$, 
then the given family $\{\,(\omega_{t},g_{t},J_{t},h_{t},\theta_{t})\mid t\in B\,\}$ and $\nabla_{0}$ should satisfy, for all $t\in B$, 
\begin{equation}\label{DKM2}
\left\{
\begin{aligned}
&[(F_{\nabla_{0}})^{(0,2)_{t}}]=0\quad\mbox{in }H^{0,2}_{\bar{\partial}_{t}},\\
&\left[\mathop{\mathrm{Im}}\left(e^{-\sqrt{-1}\theta_{t}}(\omega_{t}+F_{\nabla_{0}})^{n}\right)\right]=0\quad\mbox{in }H^{2n}_{dR}, 
\end{aligned}
\right.
\end{equation}
where $(0,2)_{t}$ is the $(0,2)$-part with respect to $J_{t}$ and 
$H^{0,2}_{\bar{\partial}_{t}}:=H^{0,2}_{\bar{\partial}_{t}}(X,J_{t})$ is 
the Dolbeault cohomology defined by the complex structure $J_{t}$. 
\end{proposition}

\begin{proof}
Define $f_{t}\in \Om^{0}$ by $h_{t}=e^{2f_{t}}h_0$. 
Then, we can easily check that 
\[
\widetilde{\nabla}_{t}:=\nabla_{0}+ df_{t}\cdot \mathrm{id}_{L}
\]
is a Hermitian connection of $(L,h_{t})$. 
Thus, there exists $a_{t}\in\Omega^{1}$ such that $\nabla_{t}=\widetilde{\nabla}_{t}+\sqrt{-1} a_{t}\cdot \mathrm{id}_{L}$. 
Then, from 
$F_{\nabla_{t}}=F_{\widetilde{\nabla}_{t}}+\sqrt{-1}da_{t} = F_{\nabla_{0}}+\sqrt{-1}da_{t}$, it follows that 
\[
(\omega_{t}+F_{\nabla_{t}})^{n}-(\omega_{t}+F_{\nabla_{0}})^{n}\in d\Omega^{2n-1}_{\mathbb{C}}. 
\]
Multiplying $e^{-\sqrt{-1}\theta_{t}}$ and taking the imaginary part, we have 
\begin{equation}\label{DKM0.5}
\mathop{\mathrm{Im}}\left(e^{-\sqrt{-1}\theta_{t}}(\omega_{t}+F_{\nabla_{t}})^{n}\right)-\mathop{\mathrm{Im}}\left(e^{-\sqrt{-1}\theta_{t}}(\omega_{t}+F_{\nabla_{0}})^{n}\right)\in d\Omega^{2n-1}. 
\end{equation}
Noting that the first term is zero since $\n_t$ is a dHYM connection with phase $e^{\sqrt{-1}\theta_{t}}$ 
with respect to $\omega_{t}$, 
we have the second equality in \eqref{DKM2}. 
Similarly, from $F_{\nabla_{t}}=F_{\nabla_{0}}+\sqrt{-1}da_{t}$, it follows that 
\begin{equation}\label{DKM1.3}
(F_{\nabla_{t}})^{(0,2)_{t}}-(F_{\nabla_{0}})^{(0,2)_{t}}=\sqrt{-1}(da_{t})^{(0,2)_{t}}
=\sqrt{-1}\bar{\partial}_{t}(a_{t}^{(0,1)_{t}})\in \bar{\partial}_{t}\Omega^{(0,1)_{t}}, 
\end{equation}
where symbols with subscript $t$ are defined by the complex structure $J_{t}$. 
Noting that the first term is zero since $\n_t$ is a dHYM connection with respect to the complex structure $J_{t}$, 
we have the first equality in \eqref{DKM2}, and the proof is completed. 
\end{proof}

We will show that \eqref{DKM2} is also a sufficient condition. 
For each $t\in B$, denote by $\mathcal{M}_{t}$ the moduli space of dHYM connections with phase $e^{\sqrt{-1}\theta_{t}}$ of $(L,h_{t})$ with respect to 
$(\omega_{t},g_{t},J_{t})$ and put 
\[
\Mm_{B'} = \bigcup_{t\in B'}\Mm_{t}, 
\]
for any subset $B'\subset B$. 

\begin{theorem}\label{deformdHYM}
Suppose that $\n_0$ is a dHYM connection of $(L,h_{0})$ on $(X,\omega_{0},g_{0},J_{0})$ with phase $e^{\sqrt{-1}\theta_{0}}$. 
If $\{\,(\omega_{t},g_{t},J_{t},h_{t},\theta_{t})\mid t\in B\,\}$ and $\nabla_0$ satisfy \eqref{DKM2}, 
then there exists an open set $B'\subset B$ containing $0$ such that $\mathcal{M}_{B'}$ is a $T^{b^{1}}$-bundle over $B'$. 
Especially, there exists a deformation  $B'\ni t\mapsto \nabla_{t}\in \mathcal{M}_{t}$ of $\nabla_0$ along $(\omega_{t},g_{t},J_{t},h_{t},\theta_{t})$. 
\end{theorem}

Thus, dHYM connections are stable under small deformations of the K\"ahler structure and the Hermitian metric on $L$.

\begin{proof}
Define an infinite dimensional vector bundle 
$\pi:\mathcal{E}\to \cup_{t\in B}(\Aa_{0,t} /\Gg_{U,t})$ by 
\[\pi^{-1}(t, [\nabla']):=\pi^{(0,2)_{t}} (d \Om^1)\oplus d\Omega^{2n-1}, \]
where $\Aa_{0,t}$ is the set of all Hermitian connections of $(L,h_{t})$ and 
$\Gg_{U,t}$ is the group of unitary gauge transformations of $(L,h_{t})$. 
Note that $\pi^{-1}(t, [\nabla'])$ depends only on $t \in B$.
Define a map 
$\mathcal{F}_B : \cup_{t\in B}(\Aa_{0,t} /\Gg_{U,t}) \to \Omega^{2}_{\mathbb{C}}\oplus\Omega^{2n}$ 
by 
\[
\mathcal{F}_B (t, [\n']) = 
\left(-\sqrt{-1}(F_{\n'})^{(0,2)_{t}}, \ \mathop{\mathrm{Im}}\left(e^{-\sqrt{-1}\theta_{t}}(\omega_{t}+F_{\n'})^{n}\right) \right). 
\]
Then, by \eqref{DKM0.5}, \eqref{DKM1.3} and the assumption \eqref{DKM2}, 
we see that
\[\mathcal{F}_B (t, [\n'] ) \in \pi^{-1}(t, [\n']). \] 
In other words, 
$(t,[\n']) \mapsto \mathcal{F}_B (t,[\n'])$
gives a section of 
$\pi:\mathcal{E}\to \cup_{t\in B}(\Aa_{0,t} /\Gg_{U,t})$. 
Since $\mathcal{M}_{B'}$ is the inverse image of the zero section by $\mathcal{F}_{B}$ over $B'$, simply denoted by $\mathcal{F}_{B}^{-1}(0)$, 
we will apply the implicit function theorem to $\mathcal{F}_{B}$. 
Here, we remark that a neighborhood of $(0,[\n_{0}])$ in $\cup_{t\in B}(\Aa_{0,t} /\Gg_{U,t})$ can be identified with that of $(0,0)$ in 
$B\times \Om^1_{d^{\tilde *_{\n_{0}}}}$, where $\Om^1_{d^{\tilde *_{\n_{0}}}}$ is the space of $d^{\tilde *_{\n_{0}}}$-closed 1-forms on $X$. 
Actually, defining $p:B\times \Om^1_{d^{\tilde *_{\n_{0}}}}\to \cup_{t\in B}(\Aa_{0,t} /\Gg_{U,t})$ by 
\[p(t,a):=(t,[\nabla_{0}+df_{t}\cdot \mathrm{id}_{L}+\sqrt{-1}a\cdot\mathrm{id}_{L}]_{t}), \]
where $f_{t}\in \Om^{0}$ is defined by $h_{t}=e^{2f_{t}}h_0$ and 
$[\,\cdot\,]_{t}$ means the equivalence class defined by $\Gg_{U,t}$-action, we can repeat the proof of Proposition \ref{p3.12}. 
Thus, the tangent space of $(\Aa_{0,t} /\Gg_{U,t})|_{t=0}$ at $[\n_{0}]$ can be identified with $\Om^1_{d^{\tilde *_{\n_{0}}}}$. 
Since $d^{\tilde *_{\n_{0}}} \Om^2\subset \Om^1_{d^{\tilde *_{\n_{0}}}}$, 
we can consider the restricted linearization map 
\[
D=(\delta_{(0, \n_{0})} \Ff_B) |_{ \{\, 0 \,\} \times d^{\tilde *_{\n_{0}}} \Om^2}: 
d^{\tilde *_{\n_{0}}} \Om^2 \rightarrow \pi^{0,2} (d \Om^1) \oplus d \Om^{2n-1}, 
\]
which is an isomorphism by Lemma \ref{lem:isom}. 
Then, we can apply the implicit function theorem 
(after the Banach completion) 
to the above section $\mathcal{F}_B$, 
and we see that $\mathcal{F}_{B}^{-1}(0)$ is smooth near $(0, [\nabla_{0}])$. 
The space $\R^m \oplus \widetilde \Hh_{d, \n_{0}}^{1}$ 
is the complement of $\{\, 0 \,\} \times d^{\tilde *_{\n_{0}}} \Om^2$ 
in $\R^m \times \Om^1_{d^{\tilde *_{\n_{0}}}}$, 
which implies that $\dim \Mm = m + b^1$. 
Especially, this implies that there exists an open set $B'\subset B$ containing $0$ and 
a deformation $B'\ni t \to \nabla_{t}\in \mathcal{M}_{t}$ such that $\nabla_{t}|_{t=0}=\nabla_{0}$. 
Since $\mathcal{M}_{t}\cong T^{b^{1}}$ if $\mathcal{M}_{t}\neq\emptyset$ by Theorem \ref{thm:smooth M}, 
we see that $\mathcal{M}_{B'}$ is a $T^{b^{1}}$-bundle over $B'$.

Finally, we put a remark on a way to recover the regularity of elements in $\mathcal{F}_{B}^{-1}(0)$ around $(0,[\n_{0}])$ after the Banach completion. 
Assume that the regularity of a 1-form $a$ is $C^{k,\alpha}$ with $k\geq 2$ and $\alpha\in (0,1)$ and $d^{\tilde *_{\n_{0}}} a=0$. 
Further assume that $\nabla_{a}:=\nabla+df_{t}\cdot\mathrm{id}_{L}+\sqrt{-1}a\cdot\mathrm{id}_{L}$ 
satisfies $\mathcal{F}_{B}(t,[\nabla_{a}])=0$. 
Then, by $(da)^{(0,2)_{t}}=0$ and the $C^{k,\alpha}$-version of Lemma \ref{lem:ker F02}, we see that $a=b+(d_{c})_{t}f$ for 
some $C^{k,\alpha}$ closed 1-form $b$ and $C^{k-1,\alpha}$ function $f$. 
Then, $f$ should satisfies 
$\mathop{\mathrm{Im}}(e^{-\sqrt{-1}\theta_{t}}(\omega_{t}+F_{\n_{a}})^{n})=0$ 
which is equivalent to 
\[\Theta_{t}(f):=\arg\left(\left(\omega_{t}+F_{\nabla}+\sqrt{-1}d(d_{c})_{t}f)\right)^{n}/\omega_{t}^{n}\right)=\theta_{t}. \]
Then, as we see in the proof of Theorem \ref{thm:smooth M}, since the linearization of $\Theta_{t}$ with respect to $f$ is elliptic, 
$f$ is actually smooth by the standard Schauder theory. 
Let $b_{\Hh}$ be the harmonic part of $b$ which is of course smooth. 
Then, $\tilde{a}:=b_{\Hh}+(d_{c})_{t}f$ is a smooth 1-form satisfying $\mathcal{F}_{B}(t,[\nabla_{\tilde{a}}])=0$. 
This means that we can replace the representative of $[\nabla_{a}]$ with a smooth one and the proof is completed. 
\end{proof}
\appendix
\section{Basic identities and the Hodge decomposition}\label{AppA}
In this appendix, we collect some basic definitions and equations 
which are used throughout this paper. 
\SkipTocEntry \subsection{The Hodge-$\ast$ operator}
Let $V$ be an $n$-dimensional oriented real vector space with a scalar product $g$. 
By an abuse of notation, we also denote by $g$ the 
induced inner product on $\Lambda^k V^*$ from $g$. 
Let $\ast$ be the Hodge-$\ast$ operator.
The following identities are frequently used throughout this paper. 

For $\alpha, \beta \in \Lambda^k V^*$ and $v \in V$, we have 
\[
\begin{aligned}
\ast^2|_{\Lambda^k V^*} &= (-1)^{k(n-k)} {\rm id}_{\Lambda^k V^*}, & 
g(\ast \alpha, \ast \beta) &= g(\alpha, \beta), \\ 
i(v) \ast \alpha &= (-1)^k \ast (v^\flat \wedge \alpha), & 
\ast( i(v) \alpha) &= (-1)^{k+1} v^\flat \wedge \ast \alpha. 
\end{aligned}
\]
\SkipTocEntry \subsection{The $d^c$ operator}
Let $(X, J)$ be a complex manifold with a Hermitian metric $g$. 
For a differential form $\alpha \in \Om^\bullet$, set 
\[
J \alpha := \alpha(J (\,\cdot\,), \cdots, J (\,\cdot\,)).
\]
Then, the complex differential $d_c$ is defined by 
\[
d_c = J^{-1} \circ d \circ J = \i (\bp - \p). 
\]
In particular, for a function $f \in \Om^0$, we have
$d_c f = J^{-1} df = - df (J(\,\cdot\,))$. 
The formal adjoint $d_c^*$ of $d_c$ is given by 
\begin{equation}\label{A1.5?}
d_c^* = J^{-1} \circ d^* \circ J = - * d_c *. 
\end{equation}
In particular, for a 1-form $\alpha \in \Om^1$, we have
\[d_c^* \alpha = d^* (J \alpha) = d^* (\alpha (J(\,\cdot\,))). \]
It is immediate to see that $2 \sqrt{-1} \p \bp = d d_c$. 

We use the following in Subsection \ref{sec:prel dHYM}. 
\begin{lemma} \label{lem:gen id}
Define the associated 2-form $\om$ by $\om=g(J(\,\cdot\,),\,\cdot\,)$. 
Then, for any 1-form $\alpha \in \Om^1$, we have 
$$
\om^{n-1} \wedge \alpha = (n-1)! * J \alpha. 
$$
\end{lemma}

\begin{proof}
Take a vector field $v$ such that $i(v) \om = \alpha$. 
Note that $J \alpha = g(J v, J \,\cdot\,) = g(v, \,\cdot\,)$. 
Then, we have 
\begin{equation*}
\om^{n-1} \wedge \alpha = i(v) (\omega^n/n) = (n-1)! i(v) {\rm vol}
= (n-1)! * J \alpha. 
\end{equation*}
\end{proof}
\SkipTocEntry \subsection{The Hodge decomposition}
Let $(X, g)$ be a compact oriented Riemannian manifold. 
By the Hodge decomposition, for any $k \geq 0$, 
\begin{equation}\label{eq:Hodge d}
\begin{aligned}
\Om^k = \Hh^k_d \oplus d \Om^{k-1} \oplus d^* \Om^{k+1}, 
\end{aligned}
\end{equation}
where $\Hh^k_d$ is the space of harmonic $k$-forms. 
Denote by $\Delta = d d^* + d^* d$ the Laplacian with respect to $g$. 

Let $(X,J)$ be a compact complex manifold with a Hermitian metric $g$. 
Then, 
$\{\,d_c: \Omega^{k} \to \Omega^{k+1}\,\}_{k \geq 0}$ is an elliptic complex. 
Indeed, 
the principal symbol of $d_{c}$ is given by 
\begin{equation}\label{profdc}
(\sigma_{\xi}(d_{c}))(\alpha)=J^{-1}(\xi\wedge J\alpha)
\end{equation}
for $x \in X, \xi \in T_x^*X$ and $\alpha \in \Lambda^k T^*_x X$. 
If $(\sigma_{\xi}(d_{c}))(\alpha)=0$ for $\xi \neq 0$, there exists $\beta \in \Lambda^{k-1} T^*_x X$ 
such that $J \alpha = \xi \wedge \beta$. 
Then, $\alpha = J^{-1} (\xi \wedge \beta) = J^{-1} (\xi \wedge J (J^{-1} \beta)) 
\in \Im(\sigma_{\xi}(d_{c}))$. 

Then, the Hodge Theory is applicable for the $d_c$-Laplacian 
$\Delta_{d_c} = d_c d_c^* + d_c^* d_c$, and we have 
\begin{equation}\label{eq:Hodge dc2}
\begin{aligned} 
\Om^k = \Hh^k_{d_c} \oplus d_c \Om^{k-1} \oplus d_c^* \Om^{k+1}, 
\end{aligned}
\end{equation}
where $\Hh^k_{d_c}= \{\,\alpha\in\Omega^{k} \mid d_{c}\alpha=0\mbox{ and }d_{c}^{*}\alpha=0\,\}$. 

For the rest of the section, we suppose that $(X, J, g)$ is a compact K\"ahler manifold. 
It is well-known that 
\[
\Delta = \Delta_{d_c}, \quad 
d d_c = -d_c d, \quad
d^* d_c = - d_c d^*, \quad
d_c^* d = - d d_c^*. 
\]
Then, we have $\Hh^k_{d} = \Hh^k_{d_c}$ and applying \eqref{eq:Hodge dc2} to \eqref{eq:Hodge d}, 
we have 
\[
\Om^k = \Hh^k_d \oplus d d_c \Om^{k-2} \oplus d d_c^* \Om^{k} \oplus 
d^* d_c \Om^{k} \oplus d^* d_c^* \Om^{k+2}. 
\]
Note that this is the orthogonal decomposition with respect to the $L^2$ inner product. 
By this decomposition, we immediately see 
\begin{equation}\label{eq:intersection d dc}
d \Om^k \cap d_c \Om^k = d d_c \Om^{k-1}. 
\end{equation}

Using this, we see the following. 
\begin{lemma}\label{lem:ker F02}
Define a map $\Tt: \Om^1 \rightarrow \Om^{0,2}$ by 
\[
\Tt (a) = \pi^{0,2} (d a), 
\]
where $\pi^{0,2}(d a)$ is the $(0,2)$-part of $d a \in \Om^2$. Then, 
$\ker \Tt
= Z^1 \oplus d_c \Om^0,  
$
where $Z^1$ is the space of closed 1-forms. 
\end{lemma}

\begin{proof}
Suppose that $a \in \ker \Tt$. 
Since $d a$ is a real 2-form, 
$\pi^{0,2} (d a) = 0$ if and only if $da \in \Om^{1,1}$, that is, 
$J da = da$. 
Then, by \eqref{eq:intersection d dc}, we see that 
$da \in d \Om^1 \cap d_c \Om^1 = d d_c \Om^0$, 
which implies that $a \in Z^1 \oplus d_c \Om^0$. 

Conversely, take any $a \in Z^1 \oplus d_c \Om^0$. 
Then, $d a \in d d_c \Om^0 = \i \p \bp \Om^0$,  
which implies that $da \in \Om^{1,1}$, and hence, $a \in \ker \Tt$. 
\end{proof}
\section{Basics on $G_2$-geometry}\label{sec:G2 geometry}
In this appendix, we collect some basic definitions and equations on 
$G_2$-geometry which we needed in the calculations
in this paper for the reader's convenience. 

Let $V$ be an oriented $7$-dimensional vector space. A \emph{$G_2$-structure} on $V$ is 
a 3-form $\varphi \in \Lambda^3 V^*$ such that there is a positively oriented basis 
$\{\, e_i \,\}_{i=1}^7$ of $V$ 
with the dual basis $\{\, e^i \,\}_{i=1}^7$ of $V^\ast$ satisfying 
\begin{equation} \label{varphi}
\varphi = e^{123} + e^{145} + e^{167} + e^{246} - e^{257} - e^{347} - e^{356},
\end{equation}
where $e^{i_1 \dots i_k}$ is short for $e^{i_1} \wedge \cdots \wedge e^{i_k}$. Setting $\vol := e^{1 \cdots 7}$, 
the 3-form $\varphi$ uniquely determines an inner product $g_\varphi$ via 
\begin{equation} \label{eq:form-1def}
g_\varphi(u,v)\; \vol = \dfrac16 i(u) \varphi \wedge i(v) \varphi \wedge \varphi 
\end{equation}
for $u,v \in V$. 
It follows that any oriented basis $\{\, e_i \,\}_{i=1}^7$ for which \eqref{varphi} holds is orthonormal with respect to $g_\varphi$. Thus, the Hodge-dual of $\varphi$ with respect to $g_\varphi$ is given by 
\begin{equation} \label{varphi*}
\ast \varphi = e^{4567} + e^{2367} + e^{2345} + e^{1357} - e^{1346} - e^{1256} - e^{1247}.
\end{equation}
The stabilizer of $\varphi$ is known to be the exceptional $14$-dimensional simple Lie group 
$G_2 \subset {\rm GL}(V)$. The elements of $G_2$ preserve both $g_\varphi$ and $\vol$, that is, 
$G_2 \subset {\rm SO}(V, g_\varphi)$.

We summarize important well-known facts about the decomposition of tensor products of $G_2$-modules into irreducible summands. 
Denote by $V_k$ the $k$-dimensional irreducible $G_2$-module if there is a unique such module. For instance, $V_7$ is the irreducible $7$-dimensional $G_2$-module $V$ from above, 
and $V_7^* \cong V_7$. For its exterior powers, we obtain the decompositions
\begin{equation} \label{eq:DiffForm-V7}
\begin{array}{rlrl}
\Lambda^0 V^* \cong \Lambda^7 V^* \cong V_1, \quad
& \Lambda^2 V^*  \cong \Lambda^5 V^* \cong V_7 \oplus  V_{14},\\[2mm]
\Lambda^1 V^* \cong \Lambda^6 V^* \cong V_7, \quad
& \Lambda^3 V^* \cong \Lambda^4 V^* \cong V_1 \oplus V_7 \oplus V_{27},
\end{array}
\end{equation}
where $\Lambda^k V^* \cong \Lambda^{7-k} V^*$ due to the $G_2$-invariance of the Hodge isomorphism $\ast: \Lambda^k V^* \to \Lambda^{7-k} V^*$. We denote by $\Lambda^k_\l V^* \subset \Lambda^k V^*$ the subspace isomorphic to $V_\l$. 
Let 
\[
\pi^k_\l: \Lambda^k V^* \rightarrow \Lambda^k_\l V^*
\]
be the canonical projection. 
Identities for these spaces we need in this paper are as follows. 
\begin{equation}\label{decom-L-V7}
\begin{aligned}
\Lambda^2_7 V^* =& \{\, i(u) \varphi \mid u \in V \,\}
= \{\, \alpha \in \Lambda^2 V^* \mid * (\varphi \wedge \alpha) = 2 \alpha \,\},\\
\Lambda^2_{14} V^* =& \{\, \alpha \in \Lambda^2 V^* \mid \ast \varphi \wedge \alpha = 0\,\} 
= \{\, \alpha \in \Lambda^2 V^* \mid * (\varphi \wedge \alpha) = - \alpha \,\},\\
\Lambda^3_1 V^* =& \R \varphi, \\
\Lambda^3_7 V^* =& \{\, i(u) * \varphi \in \Lambda^3 V^* \mid u \in V \,\}. 
\end{aligned}
\end{equation}
Let $S^k V_\l$ be the space of symmetric $k$-tensors on $V_\l$. 
We use the following irreducible decompositions in this paper. 
\begin{equation}\label{eq:decomp G2 symm}
\begin{aligned}
S^2 V_7 =& V_1 \oplus V_{27},\\
S^2 V_{14} =& V_1 \oplus V_{27} \oplus V_{77}, \\
V_7 \otimes V_{14} =& V_7 \oplus V_{27} \oplus V_{64}. 
\end{aligned}
\end{equation}
The following equations are well-known and useful in this paper. 

\begin{lemma} \label{lem:G2 identities}
For any $u \in V$, we have the following identities. 
\[
\begin{aligned}
\varphi \wedge i(u) * \varphi &= -4 * u^{\flat}, 
\\
* \varphi \wedge i(u) \varphi &= 3 * u^{\flat}, \\
\varphi \wedge i(u) \varphi &= 2 * (i(u) \varphi) = 2 u^{\flat} \wedge * \varphi. 
\end{aligned}
\]
\end{lemma}

The following equations are useful in Appendix \ref{sec:dDT irr decomp}. 

\begin{lemma} \label{lem:G2id high}
For any $u \in V$ and $\beta \in \Lambda^2_{14} V^*$, we have the following. 
\begin{align}
(i(u) \varphi)^3 &= 6 |u|^2 * u^\flat, \label{eq:G2id high 1}\\
(i(u) \varphi)^2 \wedge \beta &= 2 * \varphi \wedge u^\flat \wedge i(u) \beta, \label{eq:G2id high 2}\\
(i(u) \varphi) \wedge \beta^2 &= - |\beta|^2 * u^\flat + \varphi \wedge i(u) (\beta^2). \label{eq:G2id high 3}
\end{align}
\end{lemma}

\begin{proof}
Since $i(u) \varphi \in \Lambda^2_7 V^*$, we see that
\[
(i(u) \varphi)^2 \wedge \varphi 
= i(u) \varphi \wedge 2 * (i(u) \varphi)
= 2|i(u) \varphi|^2 \vol
= 6 |u|^2 \vol.
\]
Taking the interior product by $u$ of both sides, we obtain \eqref{eq:G2id high 1}. 
Similarly, taking the interior product by $u$ of both sides of
\[
\varphi \wedge i(u) \varphi \wedge \beta= 0, 
\]
we obtain 
\[
\begin{aligned}
0
= (i(u) \varphi)^2 \wedge \beta - \varphi \wedge i(u) \varphi \wedge i(u) \beta\\
= (i(u) \varphi)^2 \wedge \beta - 2 * \varphi \wedge u^\flat \wedge i(u) \beta, 
\end{aligned}
\]
which implies \eqref{eq:G2id high 2}. 
For \eqref{eq:G2id high 3}, since $\beta \in \Lambda^2_{14} V^*$, 
we have 
\[
\varphi \wedge \beta^2 = -|\beta|^2 \vol. 
\]
Hence, we obtain 
\[
(i(u) \varphi) \wedge \beta^2 - \varphi \wedge i(u) (\beta^2) =  -|\beta|^2 * u^\flat, 
\]
which implies \eqref{eq:G2id high 3}. 
\end{proof}

The following lemma is essential in the proof of Theorem \ref{thm:moduli MG2 generic}. 

\begin{lemma} \label{lem:kernel wedge}
For $u \in V$ and $\beta \in \Lambda^2_{14} V^*$, set $F = i(u) \varphi + \beta$. 
If 
\[
i(u) \beta =0, \quad F^3 \neq 0 \quad \mbox{and} \quad F \wedge \gamma =0
\]
for $\gamma \in \Lambda^2 V^*$, 
we have $\gamma=0$. 
\end{lemma}

Note that we cannot drop the assumption $F^3 \neq 0$. 
If we set $u=0, \beta = e^{23}-e^{45} \in \Lambda^2_{14} V^*$ 
and $\gamma = e^{24}+e^{35} \in \Lambda^2_{14} V^*$, 
we have $F \wedge \gamma =0$. 

\begin{proof}
Recall that every element in $\g_2$ is ${\rm Ad}(G_2)$-conjugate to an element of a Cartan subalgebra. 
Then, as in \cite[Section 2.7.2]{Bryant}, we may assume that 
\begin{align*}
\beta = \lambda_1 e^{23} + \lambda_2 e^{45} + \lambda_3 e^{67}
\end{align*}
for $\lambda_j \in \R$ such that $\lambda_1+\lambda_2+\lambda_3=0$. 
Set
$u = \sum_{j=1}^7 u^j e_j$ 
and 
$\gamma = \sum_{1 \leq i < j \leq 7} \gamma_{i j} e^{i j} 
$
for $u^j, \gamma_{i j} \in \R$. 

Suppose that $\lambda_1 \lambda_2 \lambda_3 \neq 0$.
Then, $i(u) \beta =0$ implies that $u = u^1 e_1$ 
and 
$$
F = (u^1 + \lambda_1) e^{23} + (u^1 + \lambda_2) e^{45} + (u^1 + \lambda_3) e^{67}. 
$$
Thus, $F^3 \neq 0$ if and only if 
$u^1+\lambda_j \neq 0$ for any $j=1,2,3$. 
Under this assumptions, 
it is straightforward to see that 
$F \wedge \gamma =0$ if and only if 
$$
T \left(
\begin{array}{c}
\gamma_{2 3} \\
\gamma_{4 5} \\
\gamma_{6 7} \\
\end{array}
\right) =0, \quad \mbox{where} \quad 
T=
\left(
\begin{array}{ccc}
u^1+\lambda_2 & u^1+\lambda_1 & 0 \\
u^1+\lambda_3 & 0                    & u^1+\lambda_1 \\
0                    & u^1+\lambda_3 & u^1+\lambda_2 \\
\end{array}
\right), 
$$
and $\gamma_{i j}=0$ for $(i,j) \neq (2,3), (4,5), (6,7)$.  
Since $\det T=-2 (u^1 + \lambda_1) (u^1 + \lambda_2) (u^1 + \lambda_3) \neq 0$, 
we also obtain $\gamma_{2 3} = \gamma_{4 5} = \gamma_{6 7} =0$, 
and hence, $\gamma =0$.

Next, suppose that 
$\lambda_1 \lambda_2 \lambda_3 = 0$ and $\beta \neq 0$. 
We may assume that 
$\lambda_1=0$ and set $\lambda = \lambda_2 = - \lambda_3 \neq 0$. 
Hence, $\beta = \lambda (e^{45}-e^{67})$. 
Then, $i(u) \beta =0$ implies that 
$u = u^1 e_1 + u^2 e_2 + u^3 e_3$. 
Now, define $j:V=\R^7 \to \R^7=V$ by 
$j (x^1,x^2,x^3,x^4,x^5,x^6,x^7) = (x^1,x^2,x^3,x^4,x^5,x^6, -x^7)$ 
and the ${\rm SU}(2)$-action $\rho: {\rm SU}(2) 
\to {\rm GL}(\R^7) = {\rm GL}(\R^3 \oplus \C^2)$ by 
\[
\rho (g) = 
\left(
\begin{array}{cc}
\rho_- (g)  & 0 \\
0 & g \\
\end{array}
\right), 
\]
where $\rho_-:{\rm SU}(2) \to {\rm SO}(3)$ is the double cover. 
It is known that this ${\rm SU}(2)$-action preserves $j^* \varphi$. 
Thus, the ${\rm SU}(2)$-action $j \circ \rho \circ j$ 
preserves $\varphi$. 
Since $\beta$ is invariant under this ${\rm SU}(2)$-action, 
we may further assume that $u^2=u^3=0$. 
Then, by the same argument as above, 
the proof in this case is done. 

Finally, suppose that $\beta=0$. 
Then, we may assume that $u=|u| e_1$. 
By the same argument as above, the proof is completed. 
\end{proof}

\begin{definition}
Let $X$ be an oriented 7-manifold. A \emph{$G_2$-structure} on $X$ is a $3$-form $\varphi \in \Om^3$ 
such that at each $p \in X$ there is a positively oriented basis 
$\{\, e_i \,\}_{i=1}^7$ of $T_p X$ such that $\varphi_p \in \Lambda^3 T^*_p X$ is of the form \eqref{varphi}. As noted above, $\varphi$ determines a unique Riemannian metric $g = g_\varphi$ on $X$ by \eqref{eq:form-1def}, 
and the basis $\{\, e_i \,\}_{i=1}^7$ is orthonormal with respect to $g$.
A $G_2$-structure $\varphi$ is called \emph{torsion-free} if 
it is parallel with respect to the Levi-Civita connection of $g=g_\varphi$. 
A manifold with a torsion-free $G_2$-structure is called a \emph{$G_2$-manifold}. 
\end{definition}

A manifold $X$ admits a $G_2$-structure if and only if 
its frame bundle is reduced to a $G_2$-subbundle. 
Hence, considering its associated subbundles, 
we see that 
$\Lambda^* T^* X$ has the same decomposition as in \eqref{eq:DiffForm-V7}. 
The algebraic identities above also hold. 
\section{The induced $G_2$-structure from a dDT connection}\label{sec:new G2 str}
This section is completely devoted to prove Theorem \ref{thm:1+F} which is mainly used in Subsection \ref{sec:infi deform G2}. 
Throughout this section, we use the notation of Appendix \ref{sec:G2 geometry}. 
Set $V =\R^7$ with the standard basis $\{\, e_{i} \,\}_{i=1}^{7}$ 
and let $g$ be the standard scalar product on $V$. 
For a 2-form $F \in \Lambda^2 V^*$, define $F^\sharp \in {\rm End} (V)$ by 
\[
g(F^\sharp (u), v) = F(u, v)
\]
for $u,v \in V$. 
Then, 
$F^\sharp$ is skew-symmetric, and hence, 
$\det(I + F^\sharp) >0$, where $I$ is the identity matrix. 

Now, define a $G_2$-structure $\varphi_F$ by 
\[
\varphi_F := (I + F^\sharp)^* \varphi, 
\]
where $\varphi$ is the standard $G_2$-structure given by \eqref{varphi}. 
By the decomposition \eqref{decom-L-V7} with respect to $\varphi$, set 
\begin{equation} \label{eq:decomp F}
F = F_7 + F_{14} = i(u) \varphi + F_{14} \in  \Lambda^2_7 V^* \oplus \Lambda^2_{14} V^* 
\end{equation}
with $u \in V$. 
The main purpose of this appendix is to prove the following theorem. 

\begin{theorem} \label{thm:1+F}
Suppose that a 2-form $F$ satisfies $-F^3/6 + F \wedge * \varphi = 0$. 
Then, we have 
\begin{equation} \label{eq:1+F Hodge}
*_{\varphi_F} \varphi_F = (I + F^\sharp)^* * \varphi 
= \left ( 1 - \frac{1}{2} \la F^2, * \varphi \ra \right ) \cdot \left(* \varphi - \frac{1}{2}F^2 \right), 
\end{equation}
where $*_{\varphi_F}$ is the Hodge star defined by the $G_2$-structure $\varphi_F$. 
In particular, we have $1- \la F^2, * \varphi \ra/2 \neq 0$. 
\end{theorem}

For a moment, let's assume Theorem \ref{thm:1+F}. 
Then, we can define a new $G_2$-structure $\tilde \varphi_F$ by 
\[
\tilde \varphi_F=  \left | 1 - \frac{1}{2} \la F^2, * \varphi \ra \right |^{-3/4}  (I + F^\sharp)^* \varphi
\]
since $1- \la F^2, * \varphi \ra/2 \neq 0$. 
Denote by $\tilde *_F$ the Hodge star induced by $\tilde \varphi_F$. 
In general, for $c>0$ and for any $G_2$-structure $\varphi'$, it is known that 
the Hodge dual of $c^3 \varphi'$ with respect to the induced metric from $c^3 \varphi'$ 
is $c^4 *' \varphi'$, where $*'$ is the Hodge star induced from $\varphi'$. 
Applying this fact to \eqref{eq:1+F Hodge}, we obtain 
\begin{equation} \label{eq:1+F conf Hodge}
\tilde *_F \tilde \varphi_F = C \left(* \varphi - \frac{1}{2}F^2 \right), 
\end{equation}
where $C=1$ if $1- \la F^2, * \varphi \ra/2>0$ and $C=-1$ if it is negative. 

We prove Theorem \ref{thm:1+F} by the following lemmas. 
The next lemma is particularly important for the computation. 

\begin{lemma} \label{lem:1+F 00}
Suppose that a 2-form $F$ given by \eqref{eq:decomp F} 
satisfies $-F^3/6 + F \wedge * \varphi = 0$. 
Then, $i(u) F = i(u) F_{14} =0$.
\end{lemma}

\begin{proof}
First, note that $i(u) F=0$ is equivalent to $u^\flat \wedge * F=0$. 
By $F^3 = 6 F \wedge * \varphi = 6 i(u) \varphi \wedge * \varphi = 18 * u^\flat$, 
where we use Lemma \ref{lem:G2 identities}, we only have to 
show $*F^3 \wedge * F=0$. 

For any 1-form $\alpha$, we have 
\[
\alpha \wedge *F^3 \wedge * F
= *F^3 \wedge * (i(\alpha^\sharp) F)
= F^3 \wedge i(\alpha^\sharp) F
= i(\alpha^\sharp) (F^4/4).
\]
Since $F^4 = 0$, we obtain $*F^3 \wedge * F=0$. 
\end{proof}

\begin{corollary} \label{cor:1+F 00}
Suppose that a 2-form $F$ given by \eqref{eq:decomp F} 
satisfies $-F^3/6 + F \wedge * \varphi = 0$. 
Then, we have $\pi^4_7 (F^2)=0$, or equivalently $\varphi \wedge * F^2 =0$. 
\end{corollary}

\begin{proof}
By \eqref{eq:decomp G2 symm} and the Schur's lemma, 
we see that 
$\pi^4_7(F_{7}^2) = \pi^4_7(F_{14}^2) =0$, 
and hence, 
$\pi^4_7 (F^2) = 2 \pi^4_7(F_{7} \wedge F_{14})$. 
We also see that 
the space of $G_2$-equivariant linear maps from $V_7 \otimes V_{14}$ to $V_7$ 
is 1-dimensional. 
Since the maps 
$V \otimes \Lambda^2_{14} V^* \to \Lambda^4_7 V^*$ given by 
$$
v \otimes \beta \mapsto \pi^4_7(i(v) \varphi \wedge \beta) 
\qquad \mbox{and} \qquad 
v \otimes \beta \mapsto i(v) \beta \wedge \varphi 
$$
are $G_2$-equivariant, 
$\pi^4_7(i(v) \varphi \wedge \beta)$ is a constant multiple of $i(v) \beta \wedge \varphi$. 
In particular, 
$\pi^4_7(F_{7} \wedge F_{14})$ 
is a constant multiple of $i(u) F_{14} \wedge \varphi$, 
which vanishes by Lemma \ref{lem:1+F 00}. 
\end{proof}

To compute $(I + F^\sharp)^* * \varphi$, 
we first describe it in terms of $F$. 
The following holds not only for $* \varphi$ but also for any 4-form. 

\begin{lemma}
We have 
\begin{equation}\label{appB-2}
\begin{aligned}
&(I + F^\sharp)^* * \varphi \\
=& * \varphi 
- \sum_i i(e_i) F \wedge i(e_i) * \varphi \\
&+ \frac{1}{2} \sum_{i,j} i(e_i) F \wedge i(e_j) F \wedge * \varphi (e_i, e_j, \,\cdot\,, \,\cdot\,) \\
&- \frac{1}{6} \sum_{i,j,k} i(e_i) F \wedge i(e_j) F \wedge i(e_k) F \wedge * \varphi (e_i, e_j, e_k, \,\cdot\,)\\
&+ \frac{1}{24} \sum_{i,j,k,\l} i(e_i) F \wedge i(e_j) F \wedge i(e_k) F \wedge i(e_\l) F \cdot 
* \varphi (e_i, e_j, e_k, e_\l). 
\end{aligned}
\end{equation}
\end{lemma}

\begin{proof}
It is enough to prove the identity for each monomial, for example $e^{4567}$. 
Then, from $(I + F^\sharp)^*e^{j}=e^{j} - i(e_j) F$, it follows that 
\[
(I + F^\sharp)^*e^{4567}=(e^{4}-i(e_4) F) \wedge (e^{5}-i(e_5) F) \wedge 
(e^{6}-i(e_6) F) \wedge (e^{7}-i(e_7) F). 
\]
Then, this gives the desired formula for $(I + F^\sharp)^* * \varphi$.  
\end{proof}

The proof of Theorem \ref{thm:1+F} will be completed by expressing 
each term on the right hand side of \eqref{appB-2} without using $\{\, e_{i} \,\}_{i=1}^{7}$. 

\begin{lemma} \label{lem:1+F 1}
For a 2-form $F$, 
we have 
\[
\sum_i i(e_i) F \wedge i(e_i) * \varphi = 3 u^\flat \wedge \varphi. 
\]
\end{lemma}

\begin{proof}
By \eqref{eq:DiffForm-V7} and the Schur's lemma, 
the space of $G_2$-equivariant linear maps from 
$\Lambda^2 V^*$ to $\Lambda^4 V^*$ is 1-dimensional. 
Since the map $\Lambda^2 V^* \ni F \mapsto \sum_i i(e_i) F \wedge i(e_i) * \varphi \in \Lambda^4 V^*$
is $G_2$-equivariant, there exists $C \in \R$ such that 
\[
\sum_i i(e_i) F \wedge i(e_i) * \varphi = C u^\flat \wedge \varphi.
\]
Thus, it is enough to decide $C$ for some $F$ and $u$. 
Suppose that $u=e_1$ and $F_{14}=0$. Then, 
$F = i(u) \varphi = e^{23} + e^{45} +e^{67}$ and we compute 
\[
\begin{aligned}
&\sum_i i(e_i) F \wedge i(e_i) * \varphi \\
=&
e^3 \wedge i(e_2) * \varphi
- e^2 \wedge i(e_3) * \varphi
+ e^5 \wedge i(e_4) * \varphi
- e^4 \wedge i(e_5) * \varphi\\
&+ e^7 \wedge i(e_6) * \varphi
- e^6 \wedge i(e_7) * \varphi\\
=&
e^3 \wedge (e^{156} + e^{147})
- e^2 \wedge (-e^{157} + e^{146})
+ e^5 \wedge (- e^{136} - e^{127})
- e^4 \wedge (e^{137} - e^{126})\\
&+ e^7 \wedge  (e^{134} + e^{125})
- e^6 \wedge  (-e^{135} +e^{124})\\
=&
3 (-e^{1356} - e^{1347} - e^{1257} + e^{1246}). 
\end{aligned}
\]

Since $e^1 \wedge \varphi = e^1 \wedge (e^{246} - e^{257} - e^{347} - e^{356})$, 
we obtain $C=3$. 
\end{proof}

\begin{lemma} \label{lem:1+F 2}
For a 2-form $F$ given by \eqref{eq:decomp F}, 
we have 
\[
\begin{aligned}
&\sum_{i,j} i(e_i) F \wedge i(e_j) F \wedge * \varphi (e_i, e_j, \,\cdot\,, \,\cdot\,) \\
=& 
(-2 |F_7|^2 + |F_{14}|^2)* \varphi + 6 i(u) F_{14} \wedge \varphi 
+ 
(5 F_7^2 +4 F_7 \wedge F_{14} - F_{14}^2). 
\end{aligned}
\]
\end{lemma}

\begin{proof}
For the rest of this subsection, we set
\[
F_{i j} = F(e_i, e_j). 
\]
By Lemma \ref{lem:1+F 1}, we have 
\[
\sum_{i, j} i(e_i) F \wedge i(e_i) \left( i(e_j) F \wedge i(e_j) * \varphi \right) 
= 
\sum_{i} i(e_i) F \wedge i(e_i) \left( 3 u^\flat \wedge \varphi \right). 
\]
The left hand side is computed as 
\[
\sum_{i, j} \left(
-F_{i j} i(e_i) F \wedge i(e_j) * \varphi 
+ i(e_i) F \wedge i(e_j) F \wedge * \varphi (e_i, e_j, \,\cdot\,, \,\cdot\,) \right) 
\]
and the right hand side is 
\[
3 i(u) F \wedge \varphi + 3 u^\flat \wedge \sum_{i} i(e_i) F \wedge i(e_i) \varphi. 
\]
Hence, we have 
\begin{equation} \label{eq:1+F 2 1}
\begin{aligned}
&\sum_{i, j} i(e_i) F \wedge i(e_j) F \wedge * \varphi (e_i, e_j, \,\cdot\,, \,\cdot\,) \\
=&
\sum_{i, j} F_{i j} i(e_i) F \wedge i(e_j) * \varphi + 3 i(u) F \wedge \varphi 
+ 3 u^\flat \wedge \sum_{i} i(e_i) F \wedge i(e_i) \varphi. 
\end{aligned}
\end{equation}
To compute the first term of \eqref{eq:1+F 2 1}, 
we compute $\sum_{i,j} F_{i j} i(e_j) i(e_i) (F \wedge * \varphi)$ in two ways. 
We first have 
\[
\begin{aligned}
&\sum_{i,j} F_{i j} i(e_j) i(e_i) (F \wedge * \varphi)\\
=&
\sum_{i,j} F_{i j} i(e_j) \left(i(e_i)F \wedge * \varphi + F \wedge i(e_i) * \varphi \right) \\
=&
\sum_{i,j} F_{i j} \left(F_{i j} * \varphi - i(e_i) F \wedge i(e_j) * \varphi + i(e_j) F \wedge i(e_i) * \varphi + F \wedge i(e_j) i(e_i) * \varphi \right)\\
=&
\sum_{i,j} \left(F_{i j}^2 * \varphi - 2 F_{i j} i(e_i) F \wedge i(e_j) * \varphi 
+ F_{i j} F \wedge * (e^{i j} \wedge \varphi) \right). 
\end{aligned}
\]
Since $F= (1/2) \sum_{i, j} F_{i j} e^{i j}$, we have $|F|^2 = (1/2) \sum_{i, j} F_{i j}^2$. 
Thus, this is equal to 
\begin{equation}\label{appB-1}
2 |F|^2 * \varphi - 2 \sum_{i,j} F_{i j} i(e_i) F \wedge i(e_j) * \varphi 
+ 2 F \wedge * (F \wedge \varphi). 
\end{equation}
On the other hand, since $F \wedge * \varphi = i(u) \varphi \wedge * \varphi
=3 * u^\flat$ by Lemma \ref{lem:G2 identities}, we have 
\begin{equation}\label{appB-0.5}
\begin{aligned}
&\sum_{i,j} F_{i j} i(e_j) i(e_i) (F \wedge * \varphi)\\
=&
3 \sum_{i,j} F_{i j} i(e_j) i(e_i) * u^\flat
=
3 \sum_{i,j} F_{i j} * (e^{i j} \wedge u^\flat)
=
6 * (F \wedge u^\flat). 
\end{aligned}
\end{equation}
Thus, since \eqref{appB-1} is equal to the right hand side of \eqref{appB-0.5}, we obtain 
\begin{equation} \label{eq:1+F 2 2}
\sum_{i,j} F_{i j} i(e_i) F \wedge i(e_j) * \varphi
=
|F|^2 * \varphi + F \wedge * (F \wedge \varphi) -3 * (F \wedge u^\flat),
\end{equation}
and this is the first term of \eqref{eq:1+F 2 1}. 
Next, we compute the last term of \eqref{eq:1+F 2 1}. 
We have 
\[
\begin{aligned} 
\sum_i i(e_i) F \wedge i(e_i) \varphi 
&= \sum_{i, j} F_{i j} e^j \wedge i(e_i) \varphi \\
&=
\sum_{i, j} F_{i j} * (i(e_j) (e^i \wedge * \varphi)) \\
&=
- \sum_{i, j} F_{i j} * (e^i \wedge i(e_j) * \varphi) 
=
\sum_j *( i(e_j) F \wedge i(e_j) * \varphi). 
\end{aligned}
\]
Then, by Lemma \ref{lem:1+F 1}, we obtain 
\begin{equation} \label{eq:1+F 2 3}
\sum_i i(e_i) F \wedge i(e_i) \varphi = -3 i(u) * \varphi.
\end{equation}
Then, by substituting \eqref{eq:1+F 2 2} and \eqref{eq:1+F 2 3} into \eqref{eq:1+F 2 1}, we see that 
\begin{equation}\label{eq:1+F 2 4}
\begin{aligned}
&\sum_{i, j} i(e_i) F \wedge i(e_j) F \wedge * \varphi (e_i, e_j, \,\cdot\,, \,\cdot\,) \\
=&
|F|^2 * \varphi + F \wedge * (F \wedge \varphi) -3 * (F \wedge u^\flat) 
+ 3 i(u) F \wedge \varphi
-9 u^\flat \wedge i(u) * \varphi. 
\end{aligned}
\end{equation}
We can simplify this equation further. Indeed, we have 
\[
\begin{aligned}
F \wedge * (F \wedge \varphi) 
&=
(F_7 + F_{14}) \wedge (2 F_7 - F_{14})
=
2 F_7^2 +F_7 \wedge F_{14} - F_{14}^2, \\
* (F \wedge u^\flat) 
&= i(u) (* F_7 + * F_{14})\\
&= i(u) \left( \frac{1}{2} F_7 \wedge \varphi \right) - i(u) \left( F_{14} \wedge \varphi \right) \\
&=
\frac{1}{2} F_7^2 -i(u) F_{14} \wedge \varphi - F_7 \wedge F_{14}
\end{aligned}
\]
and 
\[
\begin{aligned}
u^\flat \wedge i(u) * \varphi 
&=
-i(u) (u^\flat \wedge * \varphi) + |u|^2 * \varphi \\
&=
-i(u) \left( \frac{1}{2} \varphi \wedge i(u) \varphi \right) + |u|^2 * \varphi 
=
- \frac{1}{2} F_7^2 +  |u|^2 * \varphi. 
\end{aligned}
\]
Substituting these into \eqref{eq:1+F 2 4}, we obtain 
\[
\begin{aligned}
&\sum_{i, j} i(e_i) F \wedge i(e_j) F \wedge * \varphi (e_i, e_j, \,\cdot\,, \,\cdot\,)\\
=&
|F|^2 * \varphi 
+ 
(2 F_7^2 +F_7 \wedge F_{14} - F_{14}^2)\\
&+
3 \left( - \frac{1}{2} F_7^2 + i(u) F_{14} \wedge \varphi + F_7 \wedge F_{14}\right) 
+ 
3 i(u) F \wedge \varphi
+ 
9 \left( \frac{1}{2} F_7^2 - |u|^2 * \varphi \right) \\
=&
(|F|^2 - 9 |u|^2)* \varphi + 6 i(u) F_{14} \wedge \varphi 
+ 
(5 F_7^2 +4 F_7 \wedge F_{14} - F_{14}^2). 
\end{aligned}
\]
Then, by $|F|^2 = |F_7|^2 + |F_{14}|^2$ and $|F_7|^2 = 3 |u|^2$, the proof is completed. 
\end{proof}

\begin{lemma} \label{lem:1+F 3}
Suppose that 
a 2-form $F$ given by \eqref{eq:decomp F} satisfies 
$-F^3/6 + F \wedge * \varphi =0$. 
Then, we have 
\[
\sum_{i,j,k} i(e_i) F \wedge i(e_j) F \wedge i(e_k) F \wedge * \varphi (e_i, e_j, e_k, \,\cdot\,) 
= -18 u^\flat \wedge \varphi. 
\]
\end{lemma}

\begin{proof}
We compute 
\[
J := \sum_{i,j,k} i(e_k) i(e_j) i(e_i) \left( F^3 \wedge * \varphi (e_i, e_j, e_k, \,\cdot\, ) \right)
\]
in two ways. 
Since 
\[
\begin{aligned}
i(e_k) i(e_j) i(e_i) F^3 
=& 3 i(e_k) i(e_j) \left( i(e_i) F \wedge F^2 \right) \\
=& 3 i(e_k) \left( F_{i j} F^2 -2 i(e_i) F \wedge i(e_j) F \wedge F \right) \\
=& 
3 \left( 2 F_{i j} i(e_k) F \wedge F -2 F_{i k} i(e_j) F \wedge F \right. \\
&+ \left. 2 F_{j k} i(e_i) F \wedge F -2 i(e_i) F \wedge i(e_j) F \wedge i(e_k) F
\right), 
\end{aligned}
\]
it follows that 
\[
\begin{aligned}
J =&
3 \sum_{i,j,k} \big(
6 F_{i j} i(e_k) F \wedge F \wedge * \varphi (e_i, e_j, e_k, \,\cdot\,) \\
&-2 i(e_i) F \wedge i(e_j) F \wedge i(e_k) F \wedge * \varphi (e_i, e_j, e_k, \,\cdot\,)
\big). 
\end{aligned}
\]

Note that 
the space of $G_2$-equivariant linear maps from 
$\Lambda^4 V^*$ to $\Lambda^2 V^*$ is 1-dimensional 
by \eqref{eq:DiffForm-V7} and the Schur's lemma. 
Since the map 
$$\Lambda^4 V^* \ni \gamma \mapsto 
\sum_{i,j,k} i(e_k) i(e_j) i(e_i) \gamma \wedge * \varphi (e_i, e_j, e_k, \,\cdot\,) \in \Lambda^2 V^*
$$
is $G_2$-equivariant, 
this is a constant multiple of $* \mu \left( \pi^4_7 (\gamma) \right)$, 
where $\mu : \Lambda^4_7 V^* \rightarrow \Lambda^2_7 V^*$ 
is a $G_2$-equivariant isomorphism. 
When $\gamma = F^2$, we have $\pi^4_7 (F^2) = 0$ by Corollary \ref{cor:1+F 00}.   
Then, it follows that 
\[
\begin{aligned}
0
&=\sum_{i,j,k} i(e_k) i(e_j) i(e_i) (F^2/2) \wedge * \varphi (e_i, e_j, e_k, \,\cdot\,) \\ 
&= \sum_{i,j,k} i(e_k) i(e_j) (i(e_i) F \wedge F) \wedge * \varphi (e_i, e_j, e_k, \,\cdot\,) \\
&= \sum_{i,j,k} i(e_k) (F_{i j} \wedge F - i(e_i) F \wedge i(e_j) F) \wedge * \varphi (e_i, e_j, e_k, \,\cdot\,) \\
&= 3 \sum_{i,j,k} F_{i j} i(e_k) F \wedge * \varphi (e_i, e_j, e_k, \,\cdot\,). 
\end{aligned}
\]
Hence, we obtain 
\begin{equation} \label{eq:1+F 3 1}
J= 
-6 \sum_{i,j,k}  i(e_i) F \wedge i(e_j) F \wedge i(e_k) F \wedge * \varphi (e_i, e_j, e_k, \,\cdot\,). 
\end{equation}

On the other hand, 
by $F^3 = 6 F \wedge * \varphi = 6 i(u) \varphi \wedge * \varphi
= 18 * u^\flat$, where we use Lemma \ref{lem:G2 identities}, 
we have 
\begin{equation} \label{eq:1+F 3 2}
\begin{aligned}
J
=& 18 \sum_{i,j,k} i(e_k) i(e_j) i(e_i) \left( * \varphi (e_i, e_j, e_k, u) {\rm vol} \right) \\
=& 18 \sum_{i,j,k} * \varphi (e_i, e_j, e_k, u) * e^{ijk}
= -108 * (i(u) * \varphi) 
= 108 u^\flat \wedge \varphi. 
\end{aligned}
\end{equation}
Then, by \eqref{eq:1+F 3 1} and \eqref{eq:1+F 3 2}, we obtain Lemma \ref{lem:1+F 3}. 
\end{proof}

\begin{lemma} \label{lem:1+F 4}
Suppose that a 2-form $F$ given by \eqref{eq:decomp F} satisfies $-F^3/6 + F \wedge * \varphi =0$. 
Then, we have 
\[
\begin{aligned}
&\sum_{i,j,k,\l} i(e_i) F \wedge i(e_j) F \wedge i(e_k) F \wedge i(e_\l) F \cdot 
* \varphi (e_i, e_j, e_k, e_\l) \\
=& 
- 72 i(u) (F \wedge \varphi) +  6 \la F^2, * \varphi \ra F^2. 
\end{aligned}
\]
\end{lemma}

\begin{proof} 
Since $F^4=0$, we compute 
\begin{equation}\label{appB0}
\begin{aligned}
0=&
i(e_\l)  i(e_k) i(e_j) i(e_i) (F^4/4) \\
=&
i(e_\l)  i(e_k) i(e_j) \left(i(e_i)F \wedge F^3 \right)\\
=&
i(e_\l)  i(e_k) \left(F_{i j} F^3 -3 i(e_i) F \wedge i(e_j) F \wedge F^2 \right). 
\end{aligned}
\end{equation}
The second term is computed as 
\begin{equation}\label{appB1}
\begin{aligned}
&-3 i(e_\l)  i(e_k) \left( i(e_i) F \wedge i(e_j) F \wedge F^2 \right)\\
=&
-3 i(e_\l) 
\left(F_{i k} i(e_j) F \wedge F^2 - F_{j k} i(e_i) F \wedge F^2 \right. \\
&\left. + 2 i(e_i) F \wedge i(e_j) F \wedge i(e_k) F \wedge F \right)\\
=& 
- F_{i k} i(e_\l)  i(e_j) F^3 
+ F_{j k} i(e_\l)  i(e_i) F^3 \\
&+ 6 
\left( -F_{i \l} i(e_j) F \wedge i(e_k) F \wedge F
+ F_{j \l} i(e_i) F \wedge i(e_k) F \wedge F \right. \\
&- F_{k \l} i(e_i) F \wedge i(e_j) F \wedge F\\
&\left.+ i(e_i) F \wedge i(e_j) F \wedge i(e_k) F \wedge i(e_\l) F \right). 
\end{aligned}
\end{equation}
Then, substituting \eqref{appB1} into \eqref{appB0} with noting 
\[
\begin{aligned}
6 i(e_j) F \wedge i(e_k) F \wedge F
=& 
3 i(e_j) F \wedge i(e_k) F^2\\
=& 
-3 i(e_k) \left( i(e_j) F \wedge F^2 \right) + 3 F_{j k} F^2\\
=& - i(e_k) i(e_j) F^3 +  3 F_{j k} F^2,  
\end{aligned}
\]
we see that 
\[
\begin{aligned}
0=&
F_{i j} i(e_\l)  i(e_k) F^3
- F_{i k} i(e_\l)  i(e_j) F^3 
+ F_{j k} i(e_\l)  i(e_i) F^3 \\
&+ F_{i \l} \left( i(e_k) i(e_j) F^3 - 3 F_{j k} F^2 \right) \\
&+ F_{j \l} \left( -i(e_k) i(e_i) F^3 + 3 F_{i k} F^2 \right) \\
&+ F_{k \l} \left( i(e_j) i(e_i) F^3 - 3 F_{i j} F^2 \right) \\
&+ 6 i(e_i) F \wedge i(e_j) F \wedge i(e_k) F \wedge i(e_\l) F. 
\end{aligned}
\]
Multiplying this equation by $* \varphi (e_i, e_j, e_k, e_l)$ and rearranging terms imply that 
\begin{equation}\label{appB2}
\begin{aligned}
&\sum_{i,j,k,\l} i(e_i) F \wedge i(e_j) F \wedge i(e_k) F \wedge i(e_\l) F \cdot 
* \varphi (e_i, e_j, e_k, e_\l) \\
=&
\sum_{i,j,k,\l} F_{i j} i(e_k)  i(e_\l) F^3 \cdot * \varphi (e_i, e_j, e_k, e_\l) 
+ \frac{3}{2} \sum_{i,j,k,\l} F_{i j} F_{k \l} F^2 \cdot * \varphi (e_i, e_j, e_k, e_\l). 
\end{aligned}
\end{equation}
For the second term of the right hand side of \eqref{appB2}, 
since $F = (1/2) \sum_{i,j} F_{i j} e^{i j}$, we have 
\begin{equation}\label{appB3}
\begin{aligned}
&\frac{3}{2} \sum_{i,j,k,\l} F_{i j} F_{k \l} F^2 \cdot * \varphi (e_i, e_j, e_k, e_\l) \\
=& \frac{3}{2} \sum_{i,j,k,\l} F_{i j} F_{k \l} F^2  \la e^{i j k \l}, * \varphi \ra
= 6 \la F^2, * \varphi \ra F^2. 
\end{aligned}
\end{equation}
For the first term of the right hand side of \eqref{appB2}, we first compute as
\[
\begin{aligned}
\sum_{i,j} F_{i j} * \varphi (e_i, e_j, e_k, e_\l) 
=& \sum_{i,j} F_{i j} \la e^{i j k \l}, * \varphi \ra
= 2 \la F \wedge e^{k \l}, * \varphi \ra\\
=& 2 * (F \wedge e^{k \l} \wedge \varphi)
= 2 \la e^{k \l}, *(F \wedge \varphi) \ra. 
\end{aligned}
\]
Then, by 
$F^3 = 6 F \wedge * \varphi = 6 i(u) \varphi \wedge * \varphi
= 18 * u^\flat$, where we use Lemma \ref{lem:G2 identities}, 
we obtain 
\begin{equation}\label{appB4}
\begin{aligned}
\sum_{i,j,k,\l} F_{i j} i(e_k)  i(e_\l) F^3 \cdot * \varphi (e_i, e_j, e_k, e_\l) 
=& 
- 36 \sum_{k,\l} \la e^{k \l}, *(F \wedge \varphi) \ra * (e^{k \l} \wedge u^\flat)\\
=&
- 72 * (u^\flat \wedge *(F \wedge \varphi))\\
=&
- 72 i(u) (F \wedge \varphi). 
\end{aligned}
\end{equation}
Substituting \eqref{appB3} and \eqref{appB4} into \eqref{appB2} deduces the desired formula. 
\end{proof}

\begin{proof}[Proof of Theorem \ref{thm:1+F}]
By Lemmas \ref{lem:1+F 00}, \ref{lem:1+F 1}, \ref{lem:1+F 2}, \ref{lem:1+F 3}, \ref{lem:1+F 4}
and \eqref{appB-2}, we have 
\[
\begin{aligned}
(I + F^\sharp)^* * \varphi 
=& * \varphi -3 u^\flat \wedge \varphi \\
&+ \frac{1}{2} \left((-2 |F_7|^2 + |F_{14}|^2)* \varphi 
+ 
(5 F_7^2 +4 F_7 \wedge F_{14} - F_{14}^2) \right) \\
&+ \frac{1}{6} \cdot 18 u^\flat \wedge \varphi 
+ \frac{1}{24} \left(-72 F \wedge i(u) \varphi 
+  6 \la F^2, * \varphi \ra F^2 \right) \\
=&
\left( - |F_7|^2 + \frac{1}{2} |F_{14}|^2 +1 \right) * \varphi 
+ \frac{1}{4} \la F^2, * \varphi \ra F^2 \\
&+ \frac{1}{2} (5 F_7^2 +4 F_7 \wedge F_{14} - F_{14}^2) -3 F \wedge F_7. 
\end{aligned}
\]
Since 
\[
\begin{aligned}
&\la F^2, * \varphi \ra 
= * (F^2 \wedge \varphi) 
=* (F \wedge * (2 F_7 - F_{14}))
= 2|F_7|^2 - |F_{14}|^2, \\
&\frac{1}{2} (5 F_7^2 +4 F_7 \wedge F_{14} - F_{14}^2) -3 F \wedge F_7\\
=&
- \frac{1}{2} F_7^2 - F_7 \wedge F_{14} - \frac{1}{2} F_{14}^2 
= - \frac{1}{2} (F_7 + F_{14})^2
= - \frac{1}{2} F^2, 
\end{aligned}
\]
we obtain \eqref{eq:1+F Hodge}. 
\end{proof}
\section{The proof of Proposition \ref{prop:dDT norm}}
\label{sec:dDT irr decomp}
In this section, we prove Proposition \ref{prop:dDT norm} 
which is equal to the following Corollary \ref{cor:dDT estimate}. 
The computation is pointwise, so we work in the setting of Appendix \ref{sec:new G2 str}. 
Decompose a 2-form $F \in \Lambda^2 V^*$ as in \eqref{eq:decomp F}. 
We first give another description of the defining equation of the dDT connection. 

\begin{proposition} \label{prop:decomp dDTeq}
We have 
\begin{equation} \label{eq:decomp dDTeq}
\begin{aligned}
-\frac{1}{6} F^3 + F \wedge * \varphi =& 
\left( 3 - |u|^2 + \frac{1}{2} |F_{14}|^2 \right) * u^\flat - \frac{1}{6} F_{14}^3 \\
&- * \varphi \wedge u^\flat \wedge i(u) F_{14} - \varphi \wedge F_{14} \wedge i(u) F_{14}. 
\end{aligned}
\end{equation}
If $F$ satisfies $-F^3/6 + F \wedge * \varphi = 0$, it follows that 
\begin{equation} \label{eq:dDTeq=0}
\left( 3 - |u|^2 + \frac{1}{2} |F_{14}|^2 \right) * u^\flat = \frac{1}{6} F_{14}^3. 
\end{equation}
\end{proposition}

\begin{proof}
By Lemma \ref{lem:G2id high}, we compute 
\[
\begin{aligned}
F^3 
=& F_7^3 + 3 F_7^2 \wedge F_{14} + 3 F_7 \wedge F_{14}^2 + F_{14}^3\\
=& 
6 |u|^2 * u^\flat 
+ 6 * \varphi \wedge u^\flat \wedge i(u) F_{14} 
+3 (- |F_{14}|^2 * u^\flat + \varphi \wedge i(u) (F_{14}^2)) + F_{14}^3\\
=&
(6 |u|^2 - 3 |F_{14}|^2) * u^\flat + F_{14}^3 
+ 6 * \varphi \wedge u^\flat \wedge i(u) F_{14} + 6 \varphi \wedge F_{14} \wedge i(u) F_{14}. 
\end{aligned}
\]
This together with $F \wedge * \varphi = 3 *u^\flat$ implies \eqref{eq:decomp dDTeq}. 
The equation \eqref{eq:dDTeq=0} follows from \eqref{eq:decomp dDTeq} and Lemma \ref{lem:1+F 00}. 
\end{proof}

By Proposition \ref{prop:decomp dDTeq}, we obtain the following estimates. 

\begin{corollary} \label{cor:dDT estimate}
Suppose that $F$ satisfies $- F^3/6 + F \wedge * \varphi = 0$. Then, 
\begin{enumerate}
\item
if $F_{14}=0$, we have $F_7=0$ or $|F_7| = 3$. 
\item
We have 
\[
|F_7| \leq \sqrt{2 |F_{14}|^2 + 12} 
\cos \left( \frac{1}{3} \arccos \left( \frac{|F_{14}|^3}{(|F_{14}|^2+6)^{3/2}} \right) \right). 
\]
\end{enumerate}
\end{corollary}

\begin{proof}
Set $F_{14}=0$ in \eqref{eq:dDTeq=0}. 
Then, we have $(3-|u|^2) * u^\flat =0$. 
By the equation $|F_7| = |i(u) \varphi| = \sqrt{3} |u|$, we obtain (1). 

Next, we prove (2). 
By \cite[(2.22)]{Bryant}, we have 
\[
|F_{14}^3| \leq \frac{\sqrt{6}}{3} |F_{14}|^3. 
\]
Taking absolute values on both sides of \eqref{eq:dDTeq=0}, we have 
\[
\left( |u|^2 - \frac{1}{2} |F_{14}|^2 -3 \right) |u| \leq 
\left| 3-|u|^2 + \frac{1}{2} |F_{14}|^2 \right| |u| \leq \frac{\sqrt{6}}{18} |F_{14}|^3, 
\]
and hence, 
\[
|u|^3 - \left( \frac{1}{2} |F_{14}|^2 + 3 \right) |u| - \frac{\sqrt{6}}{18} |F_{14}|^3 \leq 0.
\]
Thus, if we define a cubic polynomial $f(x)$ (with a parameter $\lambda \geq 0$) by 
\[
f(x) = x^3 - \left( \frac{1}{2} \lambda^2 + 3 \right) x - \frac{\sqrt{6}}{18} \lambda^3, 
\]
$x =|u|$ satisfies $f(x) \leq 0$ for $\lambda =|F_{14}|$. 
By the Vi\`ete's formula for a cubic equation, 
the largest solution of $f(x)=0$ is given by 
\[
x = x_0 := \sqrt{ \frac{2}{3} (\lambda^2 + 6)} 
\cos \left( \frac{1}{3} \arccos \left( \frac{\lambda^3}{(\lambda^2+6)^{3/2}} \right) \right). 
\]
Hence, we see that 
$f(x) \leq 0$ implies that $x \leq x_0$. 
This together with the equation $|F_7| = |i(u) \varphi| = \sqrt{3} |u|$ implies (2). 
\end{proof}
\section{Notation}\label{notation-list}
We summarize the notation used in this paper. 
We use the following for  an oriented Riemannian manifold $(X, g)$. 

\vspace{2ex}
\begin{center}
\begin{tabular}{|l|l|}
\hline
Notation                        & Meaning \\ \hline \hline
$i(\,\cdot\,)$                & The interior product \\
$\Gamma(X, E)$            & The space of all smooth sections of a vector bundle $E \rightarrow X$\\
$\Om^k$                       & $\Om^k = \Om^k (X) = \Gamma (X, \Lambda^k T^*X)$ \\
$\Om^k_{\C}$               & $\Om^k_{\C} = \Gamma (X, \Lambda^k T^*X \otimes \C)$ \\
$b^k$                            & The $k$-th Betti number of $X$  \\
$H^k_{dR}$                     & The $k$-th de Rham cohomology \\
$H^k (\#)$                      & The $k$-th cohomology of a complex $(\#)$ \\
$Z^1$                            & The space of closed 1-forms \\
$v^{\flat} \in T^* X$        & $v^{\flat} = g(v, \,\cdot\,)$ for $v \in TX$ \\ 
$\alpha^{\sharp} \in TX$  & $\alpha = g(\alpha^{\sharp}, \,\cdot\,)$ for $\alpha \in T^* X$ \\
$\vol$                            & The volume form induced from $g$ \\
\hline
\end{tabular}
\end{center}
\vspace{2ex}

When $(X, J)$ is a complex manifold, we use the following. 

\vspace{2ex}
\begin{center}
\begin{tabular}{|l|l|}
\hline
Notation                         & Meaning \\ \hline \hline
$\Lambda^{p,q}$             & $\Lambda^{p,q} = \Lambda^p (T^{1,0} X)^* \otimes \Lambda^q (T^{0, 1} X)^*$\\
$\Om^{p,q}$                   & $\Om^{p,q} = \Gamma (X, \Lambda^{p,q})$  \\
$H^{p,q}_\bp$                  & The Dolbeault cohomology of type $(p,q)$\\
\hline
\end{tabular}
\end{center}
\vspace{2ex}

When $X$ is a manifold with a $G_2$-structure, we use the following. 

\vspace{2ex}
\begin{center}
\begin{tabular}{|l|l|}
\hline
Notation                         & Meaning \\ \hline \hline
$\Lambda^k_\ell T^*X$       & \hspace{-1ex}\begin{tabular}{l}The subbundle of $\Lambda^k T^*X$ corresponding to an\\ $\ell$-dimensional irreducible subrepresentation\end{tabular}\\
$\Om^k_\ell$                     &  $\Om^k_\ell = \Gamma (X, \Lambda^k_\ell T^*X)$ \\
$\pi^k_\ell$                        & The projection $\Lambda^k T^*X \rightarrow \Lambda^k_\ell T^*X$ or 
                                           $\Om^k \rightarrow \Om^k_\ell$ \\
\hline
\end{tabular}
\end{center}

\end{document}